\documentclass[11pt,reqno]{amsart}
\usepackage{amsmath}
\usepackage{amsthm}
\usepackage{amsfonts}
\usepackage{amssymb}
\usepackage{tikz}
\usepackage{float}
\usepackage[all]{xy}
\usepackage{hyperref}
\usepackage{enumerate}
\usepackage{enumitem}
\usepackage[english]{babel}

\newcommand{\op}{\ensuremath{^{\mathrm{op}}}}

\newcommand{\bA}{\mathbb{A}}

\newcommand{\bZ}{\mathbb{Z}}

\newcommand{\sfS}{\mathsf{S}}

\newcommand{\cA}{{\mathcal A}}
\newcommand{\cB}{{\mathcal B}}
\newcommand{\cC}{{\mathcal C}}
\newcommand{\cD}{{\mathcal D}}
\newcommand{\cE}{{\mathcal E}}
\newcommand{\cF}{{\mathcal F}}
\newcommand{\cG}{{\mathcal G}}

\newcommand{\cP}{{\mathcal P}}

\newcommand{\cX}{{\mathcal X}}
\newcommand{\cY}{{\mathcal Y}}

\newcommand{\ob}{\operatorname{ob}}
\newcommand{\md}{\operatorname{mod}}
\newcommand{\Md}{\operatorname{Mod}}

\newcommand{\im}{\operatorname{im}}

\newcommand{\Proj}{\operatorname{Proj}}
\newcommand{\Gproj}{\cG\cP}

\newcommand{\Mon}{\operatorname{Mon}}
\newcommand{\res}{\operatorname{res}}

\newcommand{\unit}[2]{\alpha^{#1\dashv #2}}
\newcommand{\counit}[2]{\beta^{#1\dashv #2}}
\newcommand{\adjiso}[2]{\phi^{#1\dashv #2}}

\newcommand{\relGproj}[2]{\cG\cP_{#1}(#2 )}

\newcommand{\Discr}[2]{\operatorname{Discr}_{#1}(#2)}
\newcommand{\Mor}{\operatorname{Mor}}

\newcommand{\emphbf}[1]{\emph{\textbf{#1}}}
\newcommand{\pdim}{\operatorname{pdim}}

\newcommand{\Gpdim}{\operatorname{G.pdim}}

\newcommand{\glGpdim}{\operatorname{gl.Gpdim}}
\newcommand{\glGidim}{\operatorname{gl.Gidim}}

\newcommand{\Hom}{\operatorname{Hom}}

\newcommand{\Ext}{\operatorname{Ext}}

\newcommand{\Tor}{\operatorname{Tor}}

\newcommand{\Ker}{\operatorname{Ker}}
\newcommand{\Coker}{\operatorname{Coker}}

\newcommand{\bul}{\bullet}

\usetikzlibrary{matrix,arrows}

\title{Gorenstein projective objects in functor categories}

\date{\today}


\author{Sondre Kvamme}

\address{Laboratoire de Math\'ematiques d'Orsay, Univ. Paris-Sud, CNRS, Universit\'e Paris-Saclay, 91405 Orsay, France}

\email{sondre.kvamme@u-psud.fr}

\begin{document}

\newtheorem{Theorem}[equation]{Theorem}
\newtheorem{Lemma}[equation]{Lemma}
\newtheorem{Corollary}[equation]{Corollary}
\newtheorem{Proposition}[equation]{Proposition}

\theoremstyle{definition}
\newtheorem{Definition}[equation]{Definition}
\newtheorem{Example}[equation]{Example}
\newtheorem{Remark}[equation]{Remark}
\newtheorem{Setting}[equation]{Setting}

\thanks{This is part of the authors PhD thesis. The author thanks Jan Geuenich and Julian K\"ulshammer for helpful comments on a previous version of this paper, and the anonymous referee for useful suggestions and comments. The work was made possible by the funding provided by the \emph{Bonn International Graduate School in Mathematics}.}


\begin{abstract}
Let $k$ be a commutative ring, let $\cC$ be a small, $k$-linear, Hom-finite, locally bounded category, and let $\cB$ be a $k$-linear abelian category. We construct a Frobenius exact subcategory $\Gproj(\relGproj{P}{\cB^{\cC}})$ of the functor category $\cB^{\cC}$, and we show that it is a subcategory of the Gorenstein projective objects $\Gproj(\cB^{\cC})$ in $\cB^{\cC}$. Furthermore, we obtain criteria for when $\Gproj(\relGproj{P}{\cB^{\cC}})=\Gproj(\cB^{\cC})$. We show in examples that this can be used to compute $\Gproj(\cB^{\cC})$ explicitly.
\end{abstract}

\maketitle

\setcounter{tocdepth}{2}
\numberwithin{equation}{section}
\tableofcontents

\section{Introduction}\label{Introduction}
In homological algebra the projective and injective modules play a central role. The analogue in Gorenstein homological algebra are the Gorenstein projective and Gorenstein injective modules. These were defined by Auslander and Bridger in \cite{AB69} for a two sided Noetherian ring, and were later extended to a general ring in \cite{EJ95}. Nowadays, the field of Gorenstein homological algebra has turned into a well-developed subject and an active area of research, see \cite{EJ11,EJ11a}. Some examples of other papers are \cite{AM02,Bel05,Bel11,BK08,BR07,BM07,EEG08,Hol04,J07,JZ00}. It has also found applications in other areas, see for example \cite{DSS17}. In particular, the Gorenstein projective modules are used when categorifying cluster algebras \cite{JKS16,NCha17,Pre17a}, and being able to describe them is therefore important.

Let $k$ be a commutative ring, let $\cB$ be a $k$-linear abelian category with enough projectives, and let $\cC$ be a small $k$-linear category. Furthermore, let $\cB^{\cC}$ denote the category of $k$-linear functors from $\cC$ to $\cB$.
\begin{Example}\label{Example:I1}
Let $\cC=k\bA_2$ where $k\bA_2$ is the $k$-linearization of the category $\bullet \to \bullet$. The category $\cB^{k\bA_2}$ can then be identified with the morphism category $\Mor(\cB)$ of $\cB$. Since $\Mor(\cB)$ is abelian and has enough projectives, it also has Gorenstein projective objects.  By Corollary 3.6 in \cite{JK11}, a morphism $B_1\xrightarrow{f}B_2$ is Gorenstein projective in $\Mor(\cB)$ if and only if $f$ is a monomorphism and $\Coker f$, $B_1$, and $B_2$ are Gorenstein projective in $\cB$. Since Gorenstein projective objects are closed under kernels of epimorphisms, this is equivalent to only requiring  $\Coker f$ and $B_2$ to be Gorenstein projective. 
\end{Example}
Motivated by this example, one can hope to describe the Gorenstein projective objects in $\cB^{\cC}$ more generally.  Several authors \cite{EEG09,EHS13,HLXZ17, LZ13,LZ17,She16} have studied this problem. However, their descriptions only hold in special cases. In \cite{HLXZ17,LZ13,LZ17,She16} they assume $k$ is a field and $\cC$ is either $kQ$ where $Q$ is a finite acyclic quiver, $kQ/I$ where $I$ is generated by monomial relations, or a finite-dimensional Iwanaga-Gorenstein algebra, while in \cite{EEG09,EHS13} they assume $k=\bZ$ and $\cC=\bZ Q$ for a left rooted quiver $Q$. The latter results has motivated Holm and J\o rgensen to give a description of cotorsion pairs in $\cB^{\bZ Q}$ from cotorsion pairs in $\cB$, see \cite{HJ16}.

We give a more systematic description of the Gorenstein projective objects in $\cB^{\cC}$, which works for any commutative base ring $k$. Since $(\cB^{\cC})\op=(\cB\op)^{\cC\op}$, the dual results for Gorenstein injective objects are obtained by considering the opposite category. We leave the explicit statements of these results to the reader.
 
The first step is to give a suitable generalization of what it means for $f$ to be a monomorphism in Example \ref{Example:I1}. For this we need to assume that $\cC$ is a locally bounded and Hom-finite category, see Definition \ref{Definition:12,5}. The evaluation functor
\[
i^*\colon \cB^{\cC} \to \prod_{c\in \cC}\cB \quad F\to (F(c))_{c\in\cC}
\]
then has a left adjoint $i_!\colon \prod_{c\in \cC}\cB\to \cB^{\cC}$. In \cite{Kva17} it was shown that there exists a \emphbf{Nakayama functor} $\nu\colon \cB^{\cC}\to \cB^{\cC}$ \emphbf{relative to} $i_!\dashv i^*$, see Definition \ref{Nakayama functor for adjoint pair}. This means that the following holds:
\begin{enumerate}
\item $\nu$ has a right adjoint $\nu^-$;
\item The composite $\nu\circ i_!$ is right adjoint to $i^*$;
\item The unit $\lambda$ of the adjunction $\nu\dashv \nu^-$ induces an isomorphism 
\[
\lambda_{i_!((B_c)_{c\in \cC})}\colon i_!((B_c)_{c\in \cC}) \to \nu^-\nu i_!((B_c)_{c\in  \cC})
\]
for all objects $(B_c)_{c\in \cC}\in \prod_{c\in \cC}\cB$. 
\end{enumerate} 

Explicitly, the Nakayama functor is given by the weighted colimit $\nu(F)= \Hom_k(\cC,k)\otimes_{\cC}F$, and in Example \ref{Example:I1} it is just the cokernel functor $\nu(B_1\xrightarrow{f}B_2)= B_2\to \Coker f$. We give another example to illustrate this definition.

\begin{Example}[Example 3.2.6 in \cite{Kva17}]\label{Example introduction}
Let $k$ be a commutative ring, let $\Lambda_1$ be a $k$-algebra which is finitely generated projective as a $k$-module, and let $\Lambda_2$ be a $k$-algebra. If we consider $\Lambda_1$ as a $k$-linear category with one object and with endomorphism ring $\Lambda_1$, we get the identification
\[
(\Lambda_1\otimes_k \Lambda_2)\text{-}\Md = (\Lambda_2\text{-}\Md)^{\Lambda_1}.
\]
In particular, we have an adjoint pair $i_!\dashv i^*$ on $(\Lambda_1\otimes_k \Lambda_2)\text{-}\Md$ and a Nakayama functor $\nu$ relative to $i_!\dashv i^*$. Explicitly, 
\begin{align*}
& i^*\colon (\Lambda_1\otimes_k \Lambda_2)\text{-}\Md\to \Lambda_2\text{-}\Md \quad i^*(M)={}_{\Lambda_2}M \\
& i_!\colon \Lambda_2\text{-}\Md\to (\Lambda_1\otimes_k \Lambda_2)\text{-}\Md \quad i_!(M)=\Lambda_1\otimes_k M \\
& \nu\colon (\Lambda_1\otimes_k \Lambda_2)\text{-}\Md\to (\Lambda_1\otimes_k \Lambda_2)\text{-}\Md \quad \nu(M)=\Hom_k(\Lambda_1,k)\otimes_{\Lambda_1}M
\end{align*}
Note that if $k$ is a field and $\Lambda_2=k$, then we just obtain the classical Nakayama functor for a finite-dimensional algebra.
\end{Example} 
We can now apply the machinery developed in \cite{Kva17}. In particular, we can define the category $\relGproj{P}{\cB^{\cC}}$ of Gorenstein $P$-projective objects where $P=i_!\circ i^*$. Explicitly, $A\in \relGproj{P}{\cB^{\cC}}$ if and only if 
\begin{enumerate}
\item The $i$th left derived functor $L_i\nu(A)$ is $0$ for all $i>0$;
\item The $i$th right derived functor $R^i\nu^-(\nu(A))$ is $0$ for all $i>0$;
\item the unit $\lambda_A\colon A\to \nu^-\nu(A)$ of the adjunction $\nu\dashv \nu^-$ is an isomorphism on $A$.
\end{enumerate} 
See Definition \ref{Definition:7} and Theorem \ref{Theorem:2.5}. In Example \ref{Example introduction} with $k$ a field and $\Lambda_2=k$ the objects in $\relGproj{P}{\cB^{\cC}}$ are precisely the ordinary Gorenstein projective modules. Also, it turns out that for $\cC=k\bA_2$ the Gorenstein $P$-projective objects are precisely the monomorphisms. More generally, for $\cC=kQ$ where $Q$ is a locally bounded acyclic quiver, the Gorenstein $P$-projective objects are precisely the monic representations, see Definition \ref{Definition:13} and Proposition \ref{Proposition:9} part \ref{Proposition:9,2}. 

The next step is to generalize the requirement in Example \ref{Example:I1} that $B_2$ and $\Coker f$ are Gorenstein projective. Since $i^* \nu(B_1\xrightarrow{f} B_2) = (B_2,\Coker f)$, a natural guess would be that the image of $i^*\circ \nu$ must be Gorenstein projective, i.e. that we should consider the category 
\[
\{F\in \cB^{\cC}\mid F\in \relGproj{P}{\cB^{\cC}} \text{ and } i^* \nu(F)\in \prod_{c\in \cC}\Gproj(\cB)\} 
\]
which we denote by $\Gproj(\relGproj{P}{\cB^{\cC}})$. We obtain the following result for this subcategory.

\begin{Theorem}[Theorem \ref{Theorem:3}]\label{Theorem:I2}
Assume $\cB$ is a $k$-linear abelian category with enough projectives and $\cC$ is a small, $k$-linear, locally bounded, and Hom-finite category. Then the subcategory $\Gproj(\relGproj{P}{\cB^{\cC}})$ is an admissible subcategory of $\Gproj(\cB^{\cC})$.
\end{Theorem}

We refer to Definition \ref{Definition:1.5} for our definition of admissible subcategory. It implies that
\[
\Gproj(\relGproj{P}{\cB^{\cC}})\subset \Gproj(\cB^{\cC})
\]
where $\Gproj(\cB^{\cC})$ denotes the category of Gorenstein projective objects in $\cB^{\cC}$. It also implies that $\Gproj(\relGproj{P}{\cB^{\cC}})$ is a Frobenius exact subcategory of $\cB^{\cC}$. In fact, Theorem \ref{Theorem:3} holds more generally for any admissible subcategory of $\prod_{c\in \cC}\Gproj(\cB)$ and any $P$-admissible subcategory of $\relGproj{P}{\cB^{\cC}}$, see Definition \ref{Definition:10}. This gives examples of other Frobenius exact categories, see Example \ref{Example:7} and \ref{Example:8}.

It remains to determine when $\Gproj(\relGproj{P}{\cB^{\cC}})=\Gproj(\cB^{\cC})$. In general, this is not true, see Example \ref{Example:11}. However, under some mild conditions the equality holds.

\begin{Theorem}[Theorem \ref{Theorem:4}]\label{Theorem:I3}
Assume $\cB$ is a $k$-linear abelian category with enough projectives and $\cC$ is a small, $k$-linear, locally bounded and Hom-finite category. If either of the following conditions hold, then $\Gproj(\relGproj{P}{\cB^{\cC}})= \Gproj(\cB^{\cC})$:
\begin{enumerate}
\item\label{Theorem:I3,1} For any long exact sequence in $\cB^{\cC}$
	\[
	0\to K\to Q_0\to Q_1\to \cdots
	\]
	with $Q_i$ projective for $i\geq 0$, we have $K\in \relGproj{P}{\cB^{\cC}}$;
	\item\label{Theorem:I3,2} If $B\in \cB$ satisfy $\Ext^1_{\cB}(B,B')=0$ for all $B'$ of finite projective dimension, then $B\in \Gproj(\cB)$.
\end{enumerate}
\end{Theorem}

Condition \ref{Theorem:I3,1} holds when $P$ is Iwanaga-Gorenstein, see Definition \ref{Definition:9} and Corollary \ref{Gorenstein adjoint pairs lifts Gorenstein projectives}.  In this case 
\[
A\in \relGproj{P}{\cB^{\cC}} \quad \quad  \text{if and only if} \quad \quad  L_i\nu(A)=0 \text{ for all }i>0
\]
 and $\relGproj{P}{\cB^{\cC}}$ is therefore particularly easy to compute. 
 
\begin{Example}\label{Example:I1 computation}
Consider $\cC=k\bA_2$ as in Example \ref{Example:I1}. In this case, $P$ is Iwanaga-Gorenstein of dimension $1$. This implies that $L_i\nu(A)=0$ for $i>1$, and hence $A\in \relGproj{P}{\cB^{k\bA_2}}$ if and only if $L_1\nu(A)=0$. If we let $A=(B_1\xrightarrow{f}B_2)$, then a simple computation shows that 
\[
L_1\nu(B_1\xrightarrow{f}B_2)= 0\to \ker f.
\]
In particular, $(B_1\xrightarrow{f}B_2)\in \relGproj{P}{\cB^{k\bA_2}}$ if and only if $f$ is a monomorphism. Since $\nu(B_1\xrightarrow{f}B_2)=B_2\to \Coker f$, we recover the description in Example \ref{Example:I1}
\end{Example} 

More generally, for any locally bounded quiver, $P$ is Iwanaga-Gorenstein of dimension less than or equal $1$. Using this, we recover the description in \cite{EHS13} and \cite{LZ13}, see Proposition \ref{Proposition:10}. We also illustrate how to compute the Gorenstein projectives for quivers with relations in Examples \ref{Example:12}, \ref{Example:13} and \ref{Example:14}. Finally, note that Condition \ref{Theorem:I3,2} of Theorem \ref{Theorem:I3} holds when $\Gpdim B<\infty$ for all $B\in \cB$, see Lemma \ref{Proj Gorenstein}. In particular, $\Gproj(\relGproj{P}{\cB^{\cC}})= \Gproj(\cB^{\cC})$ if $\cB=\md\text{-}\Lambda$ or $\Md\text{-}\Lambda$ for an Iwanaga-Gorenstein algebra $\Lambda$.

Applying Theorem \ref{Theorem:I3} to Example \ref{Example introduction} with $k$ a field, we obtain the following result. 

\begin{Theorem}[Example \ref{Example:10}]\label{Theorem:I4}
Let $k$ be a field, let $\Lambda_1$ be a finite-dimensional $k$-algebra, and let $\Lambda_2$ be a $k$-algebra. If $\Omega^{\infty}(\Lambda_1\text{-}\Md)\subset \Gproj (\Lambda_1\text{-}\Md)$ or
\begin{multline*}\Gproj(\Lambda_2\text{-}\Md)=\{M\in \Lambda_2\text{-}\Md\mid \Ext^1_{\Lambda}(M,M')=0 \\ 
\text{ for all } M' \text{ of finite projective dimension} \}
\end{multline*} then
\begin{align*}
\Gproj((\Lambda_1\otimes_k \Lambda_2)\text{-}\Md) 
= & \{M\in (\Lambda_1\otimes_k \Lambda_2)\text{-}\Md \mid \text{ } _{\Lambda_1}M\in \Gproj(\Lambda_1\text{-}\Md) \\ &\text{ and } _{\Lambda_2}(\Hom_k(\Lambda_1,k)\otimes_{\Lambda_1} M)\in \Gproj(\Lambda_2\text{-}\Md)\}.
\end{align*} 
Hence, this equality holds in particular if $\Lambda_1$ or $\Lambda_2$ is Iwanaga-Gorenstein.
\end{Theorem}
We have an analogous statement for finitely presented modules, see Example \ref{Example:9}. Finally, using the explicit description of the Gorenstein projective objects in Theorem \ref{Theorem:I3} we also obtain a partly generalization of \cite[Theorem 4.6]{DSS17}, see Theorem \ref{Theorem:6} and Remark \ref{Remark:3}.

The paper is organized as follows. In Section \ref{Section Preliminaries} we recall the notion of Nakayama functors relative to adjoint pairs and the necessary notions in Gorenstein homological algebra. We introduce $P$-admissible subcategories of $\relGproj{P}{\cA}$ in Subsection \ref{Admissible subcategories}. In Subsection \ref{Lifting admissible subcategories} we show that adjoint pairs with Nakayama functor lift admissible subcategories of Gorenstein projectives, see Theorem \ref{Theorem:3}. In Subsection \ref{Lifting Gorenstein projectives} we use Theorem \ref{Theorem:3} to lift Gorenstein projective objects, and we provide sufficient criteria for when all Gorenstein projective objects are obtained, see Theorem \ref{Theorem:4}. In Section \ref{Application to functor categories} we study the functor category $\cB^{\cC}$ in detail.  In Subsection \ref{Monic representations of a quiver} we use Theorem \ref{Theorem:4} to recover the known description of $\Gproj(\cB^{kQ})$ for $Q$ a finite acyclic quiver, and in Subsection \ref{More examples} we compute $\Gproj(\cB^{\cC})$ for other examples of $\cC$.

\subsection{Conventions}\label{Coventions}
For a ring $\Lambda$ we let $\Lambda\text{-}\Md$ ($\Lambda\text{-}\md$) denote the category of (finitely presented) left $\Lambda$-modules. We fix $k$ to be a commutative ring.  All categories are assumed to be preadditive and all functors are assumed to be additive. $\cA$ and $\cB$ always denote abelian categories, and $\cD$ and $\cE$ always denote additive categories. We let $\Proj(\cA)$ denote the category of projective objects in $\cA$.  The projective dimension of an object $A\in \cA$ is denote by $\pdim A$. If $\cB$ and $\cC$ are $k$-linear categories, then $\cB^{\cC}$ denotes the category of $k$-linear functors from $\cC$ to $\cB$. We write $F\dashv G\colon \cD\to \cE$ to denote that we have a functor $F\colon \cD\to \cE$ with right adjoint $G\colon \cE\to \cD$. In this case we let $\unit{F}{G}$ and $\counit{F}{G}$ denote the unit and counit of the adjunction, respectively. Furthermore, $\adjiso{F}{G}\colon \cE(F(D),E)\to \cD(D,G(E))$ denotes the adjunction isomorphism.  If $\sigma\colon F_1\to F_2$ is a natural transformations, then $\sigma_G\colon F_1\circ G\to F_2\circ G$ denotes the natural transformation obtained by precomposing with $G$.

\section{Preliminaries}\label{Section Preliminaries}

\subsection{Gorenstein projective objects}\label{Gorenstein projective objects}
Let $\cA$ be an abelian category. We say that $\cA$ has \emphbf{enough projectives} if for any object $A\in \cA$ there exists an object $Q\in \Proj(\cA)$ and an epimorphism $Q\to A$. 
\begin{Definition}\label{Definition:1} Assume $\cA$ has enough projectives: 
\begin{enumerate}
\item An acyclic complex of projective objects in $\cA$
\[
Q_{\bul} = \cdots \xrightarrow{f_2} Q_1\xrightarrow{f_1} Q_0 \xrightarrow{f_{0}} \cdots
\]
is called \emphbf{totally acyclic} if the complex 
\[
\cA(Q_{\bul},Q) = \cdots \xrightarrow{-\circ f_0} \cA(Q_0,Q)\xrightarrow{-\circ f_1} \cA(Q_1,Q)\xrightarrow{-\circ f_2} \cdots
\]
is acyclic for all $Q\in \Proj (\cA)$.
\item An object $A\in \cA$ is called \emphbf{Gorenstein projective} if there exists a totally acyclic complex $Q_{\bul}$ with $A=Z_0(Q_{\bullet})=\Ker f_0$. We denote the full subcategory of Gorenstein projective objects in $\cA$ by $\Gproj(\cA)$.
\end{enumerate}
\end{Definition} 

\begin{Lemma}\label{Lemma:0.1}
If $\cA$ has enough projectives, then the subcategory $\Gproj (\cA)$ is closed under extensions and direct summands
\end{Lemma}  

\begin{proof}
The fact that $\cA$ is closed under direct summands follows from Theorem 1.4(2) in \cite{Hua13}. The fact that $\Gproj(\cA)$ is closed under extensions follows from Proposition 2.13 (1) in \cite{Bel00}
\end{proof}

\begin{Definition}\label{Definition:1.5}
Assume $\cA$ has enough projectives. A full subcategory $\cF\subset \cA$ is called an \emphbf{admissible subcategory of} $\Gproj(\cA)$ if it is closed under extensions, direct summands, and satisfies the following properties:
\begin{enumerate}
	\item\label{Definition:1.5,1} $\cF$ contains the projective objects in $\cA$;
	\item\label{Definition:1.5,2} $\Ext^1(A,Q)=0$ for all $A\in \cF$ and $Q\in \Proj(\cA)$;
	\item \label{Definition:1.5,3} For all $A\in \cF$ there exists an exact sequence $0\to A'\to Q\to A\to 0$ with $A'\in \cF$ and $Q\in \Proj(\cA)$;
	\item \label{Definition:1.5,4} For all $A\in \cF$ there exists an exact sequence $0\to A\to Q\to A'\to 0$ with $A'\in \cF$ and $Q\in \Proj(\cA)$.
\end{enumerate} 
\end{Definition}

Assume $\cF$ is an admissible subcategory of $\Gproj(\cA)$. Since $\cF$ is closed under extensions, it inherits an exact structure from $\cA$ (see \cite{Bue10} for the theory of exact categories). In fact, under this exact structure $\cF$ becomes a Frobenius exact category, and the projective objects in $\cF$ are precisely the projective objects in $\cA$. 

The following result is immediate from the definition. 

\begin{Proposition}\label{Proposition:2}
Assume $\cA$ has enough projectives. The following holds:
\begin{enumerate}
	\item\label{Proposition:2,1} $\Gproj(\cA)$ is an admissible subcategory of $\Gproj(\cA)$;
	\item $\Proj(\cA)$ is an admissible subcategory of $\Gproj(\cA)$;
	\item\label{Proposition:2,2} Assume $\cF$ is an admissible subcategory of $\Gproj(\cA)$. Then $\cF\subset \Gproj(\cA)$.
\end{enumerate} 
\end{Proposition}

Recall that a full subcategory $\cX \subset\cA$ is called \emphbf{generating} if for any $A\in \cA$ there exists an object $X\in \cX$ and an epimorphism $X\to A$. A full subcategory $\cX \subset\cA$ is called \emphbf{resolving} if it is generating and closed under direct summands, extensions, and kernels of epimorphisms. Here we follow the same conventions as in \cite{Sto14}. Note that a resolving subcategory contains all the projective objects in $\cA$. 

\begin{Lemma}\label{Lemma:1} 
Assume $\cA$ has enough projectives, and let $\cF$ be an admissible subcategory of $\Gproj(\cA)$. Then $\cF$ is a resolving subcategory of $\cA$. 
\end{Lemma}

\begin{proof}
We only need to check that it is closed under kernels of epimorphisms. Let $0\to A_3\xrightarrow{f} A_2\xrightarrow{g} A_1\to 0$ be an exact sequence in $\cA$ with $A_2\in \cF$ and $A_1\in \cF$. Choose an exact sequence $0\to A\xrightarrow{i} Q\xrightarrow{p} A_1\to 0$ in $\cA$ with $Q$ projective and $A\in \cF$. Since $Q$ is projective, there exists a morphism $s\colon Q\to A_2$ satisfying $g\circ s = p$. This gives a commutative diagram
\[
\begin{tikzpicture}[description/.style={fill=white,inner sep=2pt}]
\matrix(m) [matrix of math nodes,row sep=2.5em,column sep=5.0em,text height=1.5ex, text depth=0.25ex] 
{0 & A & Q & A_1 & 0\\
 0 & A_3 & A_2 & A_1 & 0 \\};
\path[->]
(m-1-1) edge node[auto] {$$} 	    													    (m-1-2)
(m-1-2) edge node[auto] {$i$} 	    													  (m-1-3)
(m-1-3) edge node[auto] {$p$} 	    													  (m-1-4)
(m-1-4) edge node[auto] {$$} 	    													    (m-1-5)
(m-2-1) edge node[auto] {$$} 	    													    (m-2-2)
(m-2-2) edge node[auto] {$f$} 	    											(m-2-3)
(m-2-3) edge node[auto] {$g$} 	    											(m-2-4)
(m-2-4) edge node[auto] {$$} 	    													    (m-2-5)
(m-1-2) edge node[auto] {$$} 	    								    (m-2-2)
(m-1-3) edge node[auto] {$s$} 	    									  (m-2-3)
(m-1-4) edge node[auto] {$1_{A_1}$} 	    								   	(m-2-4);	
\end{tikzpicture}
\]
with exact rows, where the morphism $A\to A_3$ is induced from the commutativity of the right square. By Lemma 5.2 in \cite{Pop73} the left square is a pushforward and a pullback square, and hence we get an exact sequence
\[
0\to A\to A_3\oplus Q\to A_2\to 0.
\]
Since $\cF$ is closed under extensions and direct summands, it follows that $A_3\in \cF$.
\end{proof}
In particular, it follows that $\Gproj(\cA)$ is a resolving subcategory of $\cA$.

One can define the \emphbf{resolution dimension}  $\dim_{\cX}(A)$ of any object $A\in \cA$ with respect to a resolving subcategory $\cX$ of $\cA$, see \cite{Sto14}. It is the smallest integer $n\geq 0$ such that there exists an exact sequence 
\[
0\to X_n\to \cdots X_1\to X_0\to A\to 0 
\] 
where $X_i\in \cX$ for $0\leq i\leq n$. In this case, if 
\[
0\to X'_n\to \cdots X'_1\to X'_0\to A\to 0 
\]
is another exact sequence with  $X_i'\in \cX$ for all $0\leq i\leq n-1$, then $X'_n\in \cX$, see \cite[Proposition 2.3]{Sto14}. We write $\dim_{\cX}(A)=\infty$ if there doesn't exist such an $n$. The \emphbf{global resolution dimension} $\dim_{\cX}(\cA)$ of $\cA$ with respect to $\cX$ is the supremum of  $\dim_{\cX}(A)$ over all $A\in \cA$. Putting $\cX=\Gproj(\cA)$ we get the \emphbf{Gorenstein projective dimension} 
\[
\Gpdim (A):= \dim_{\Gproj(\cA)}(A) 
\]
and the \emphbf{global Gorenstein projective dimension}
\[
\glGpdim (\cA):= \dim_{\Gproj(\cA)}(\cA).
\]

We need the following lemma later.

\begin{Lemma}\label{Lemma:2}
Let $A_2\xrightarrow{f}A_1$ be a morphism in $\cA$ with $A_2\in \Gproj(\cA )$. Assume 
\[
\cA(A_1,Q)\xrightarrow{-\circ f} \cA(A_2,Q)
\]
is an epimorphism for all projective objects $Q\in \cA$. Then $f$ is a monomorphism.
\end{Lemma}

\begin{proof}
Let $A_2\xrightarrow{i} Q$ be a monomorphism into a projective object $Q$. By assumption, there exists a morphism $h\colon A_1\to Q$ such that $i=h\circ f$. This implies that $f$ is a monomorphism, and we are done. 
\end{proof}

\subsection{Derived functors}\label{Derived functors}

For a functor $F\colon \cD\to \cE$, we let $\im F$ denote the full subcategory of $\cE$ consisting of the objects $F(D)$ for $D\in \cD$. 

\begin{Proposition}[Proposition 3.1.4 in \cite{Kva17}]\label{Right and left derived functors}
Let $\cA$ and $\cB$ be abelian categories, and let $G\colon \cA\to \cB$ be a functor.
\begin{enumerate}
\item\label{Right and left derived functors:1} Assume $G$ is left exact, $L\dashv R\colon \cA\to \cD$ is an adjunction, and $\im R$ is a cogenerating subcategory of $\cA$. If $R\circ L$ and $G\circ R\circ L$ are exact functors, then the $i$th right derived functor $R^iG$ of G exists for all $i>0$, and $R^iG(X)=0$ for all $i>0$ and $X\in \im R$;
\item\label{Right and left derived functors:2} Assume $G$ is right exact, $L'\dashv R'\colon \cD\to \cA$ is an adjunction, and $\im L'$ is a generating subcategory of $\cA$. If $L'\circ R'$ and $G\circ L'\circ R'$ are exact functors, then the $i$th left derived functor $L_iG$ of G exists for all $i>0$, and $L_iG(X)=0$ for all $i>0$ and $X\in \im L'$.
\end{enumerate}
\end{Proposition}
We say that $R$ \emphbf{is adapted to} $G$ or $L'$ \emphbf{is adapted to} $G$ in these two cases, respectively. Note that $\im R$ is cogenerating if and only the unit of the adjunction $L\dashv R$ is a monomorphism. By the dual of \cite[Theorem IV.3.1]{MLan98} this is equivalent to $L$ being faithful. Dually, $\im L'$ is generating if and only if the counit of $L'\dashv R'$ is an epimorphism, and by \cite[Theorem IV.3.1]{MLan98} this is equivalent to $R'$ being faithful.

We need the following result.

\begin{Lemma}\label{Lemma:5**}
Let $\cA$ and $\cB$ be abelian categories, and let $\eta\colon G_1\to G_2$ and $\epsilon\colon G_2\to G_3$ be two natural transformations between functors $\cA\to \cB$.
\begin{enumerate}
\item\label{Lemma:5**:1} Assume $G_i$ is left exact for all $i$. Furthermore, assume there exists an adjunction $L\dashv R\colon \cA\to \cD$ such that $R$ is adapted to $G_i$ for all $i$. If the sequence 
\[
0\to G_1\circ R\circ L \xrightarrow{\eta_{R\circ L}}G_2\circ R\circ L \xrightarrow{\epsilon_{R\circ L}}G_3\circ R\circ L \to 0
\]
is exact, then there exists a long exact sequence
\begin{multline*}
0\to G_1 \xrightarrow{\eta} G_2 \xrightarrow{\epsilon} G_3\to L_1G_1\to L_1G_2\to L_1G_3\to L_2G_1\to \cdots
\end{multline*}
\item\label{Lemma:5**:2} Assume $G_i$ is right exact for all $i$. Furthermore, assume there exists an adjunction $L'\dashv R'\colon \cD\to \cA$ such that $L'$ is adapted to $G_i$ for $1\leq i\leq 3$. If the sequence 
\[
0\to G_1\circ L'\circ R' \xrightarrow{\eta_{L'\circ R'}}G_2\circ L'\circ R' \xrightarrow{\epsilon_{L'\circ R'}}G_3\circ L'\circ R' \to 0
\]
is exact, then there exists a long exact sequence
\begin{multline*}
\cdots \to L_2G_3\to L_1G_1\to L_1G_2\to L_1G_3 \to G_1\xrightarrow{\eta}G_2\xrightarrow{\epsilon}G_3\to 0
\end{multline*}
\end{enumerate} 
\end{Lemma}

\begin{proof}
We prove part (ii), part (i) follows dually. Let $\sfS$ be the induced comonad on $\cA$ from the adjunction $L'\dashv R'$. There is an obvious natural isomorphism $L_nG_i\cong H_n(-,G_i)$ and $G_i\cong H_0(-,G_i)$ where $H_i(-,G_i)$ is the comonad homology relative to $\sfS$ as defined in Section 1 in \cite{BB69}. The claim follows now from Section 3.2 in \cite{BB69}. 
\end{proof}

\subsection{Nakayama functor}

We need the notion of Nakayama functor relative to adjoint pairs which was introduced in \cite{Kva17}.

\begin{Definition}\label{Nakayama functor for adjoint pair}
Let $f^*\colon \cA\to \cD$ be a faithful functor with left adjoint $f_!\colon \cD\to \cA$. A \emphbf{Nakayama functor} relative to $f_!\dashv f^*$ is a functor $\nu\colon \mathcal{A}\to \mathcal{A}$ with  a right adjoint $\nu^-$ satisfying:
\begin{enumerate}
\item $\nu\circ f_!$ is right adjoint to $f^*$;
\item The unit of $\nu\dashv\nu^-$ induces an isomorphism $f_!\xrightarrow{\cong} \nu^-\circ \nu \circ f_!$ when precomposed with $f_!$.
\end{enumerate}
\end{Definition}
We let $\lambda\colon 1_{\cA}\to \nu^-\circ \nu$ and $\sigma\colon \nu\circ \nu^-\to 1_{\cA}$ denote the unit and counit of the adjunction $\nu\dashv \nu^-$. We also fix the notation $f_*:=\nu\circ f_!$, $P:=f_!\circ f^*$ and $I:=f_*\circ f^*$. Note that we have adjunctions 
\[
f^*\circ \nu \dashv f_!\dashv f^*\dashv f_*\dashv f^*\circ \nu^-. 
\]
We call summands of objects $P(A)$ for $P$-\emphbf{projective} and summands of objects $I(A)$ for $I$-\emphbf{injective}. By the triangle identities the $P$-projectives and $I$-injectives are precisely the summands of objects of the form $f_!(D)$ and $f_*(D)$ for $D\in \cD$, respectively. Since $P$, $\nu\circ P=I$, and $\nu^-\circ I\cong P$ are exact, it follows from Proposition \ref{Right and left derived functors} that $f_!$ is adapted to $\nu$ and $f_*$ is adapted to $\nu^-$. In particular, the derived functors $L_i\nu$ and $R^i\nu^-$ exist for all $i>0$.   

\begin{Definition}[Definition 4.1.1 in \cite{Kva17}]\label{Definition:7}
Assume $\nu$ is a Nakayama functor relative to $f_!\dashv f^*\colon \cD\to \cA$. An object $X\in \cA$ is called \emphbf{Gorenstein} $P$-\emphbf{projective} if there exists an exact sequence 
\[
A_{\bullet}=\cdots \xrightarrow{f_{2}} A_{1}\xrightarrow{f_{1}} A_0\xrightarrow{f_0} A_{-1}\xrightarrow{f_{-1}} \cdots
\] 
with $A_i\in \cA$ being $P$-projective for all $i\in \bZ$, such that the sequence
\[
\nu(A_{\bullet})=\cdots \xrightarrow{\nu(f_{2})} \nu(A_{1})\xrightarrow{\nu(f_{1})} \nu(A_0)\xrightarrow{\nu(f_0)} \nu(A_{-1})\xrightarrow{\nu(f_{-1})} \cdots
\] 
is exact, and with $Z_0(A_{\bullet})=\Ker f_0=X$. The subcategory of $\cA$ consisting of all Gorenstein $P$-projective objects is denoted by $\relGproj{P}{\cA}$.
\end{Definition}

\begin{Proposition}\label{Closure of Gorenstein P-projective}
Assume $\nu$ is a Nakayama functor relative to the adjunction $f_!\dashv f^*\colon \cD\to \cA$. The following holds:
\begin{enumerate}
\item $\relGproj{P}{\cA}$ is a resolving subcategory of $\cA$;
\item Assume $i\colon A_2\to A_1$ is a morphism such that $\nu(i)$ is a monomorphism and $A_1,A_2\in \relGproj{P}{\cA}$. Then $i$ is a monomorphism and $\Coker i\in \relGproj{P}{\cA}$.
\end{enumerate}
\end{Proposition}
\begin{proof}
This follows from \cite[Proposition 4.1.5]{Kva17} and \cite[Lemma 4.1.6]{Kva17}.
\end{proof}

\begin{Definition}[Definition 4.2.1 in \cite{Kva17}]\label{Definition:9}
Assume $\nu$ is a Nakayama functor relative to $f_!\dashv f^*\colon \cD\to \cA$. We say that $P$ is \emphbf{Iwanaga-Gorenstein} if there exists an integer $n\geq 0$ such that $L_i\nu=0$ and $R^i\nu^-=0$ for all $i>n$.
\end{Definition}

\begin{Theorem}[Theorem 4.2.6 in \cite{Kva17}]\label{Theorem:2}
Assume $\nu$ is a Nakayama functor relative to $f_!\dashv f^*\colon \cD\to \cA$, and that $P$ is Iwanaga-Gorenstein. Then the following numbers coincide: 
\begin{enumerate}
\item\label{Theorem:2,1} $\dim_{\relGproj{P}{\cA}}(\cA)$;
\item\label{Theorem:2,2} The smallest integer $r$ such that $L_i\nu=0$ for all $i>r$;
\item\label{Theorem:2,3} The smallest integer $s$ such that $R^i\nu^-=0$ for all $i>s$.
\end{enumerate}
\end{Theorem}
If this common number is $n$ we say that $P$ is $n$\emphbf{-Gorenstein}. We also say that $n$ is the Gorenstein dimension of $P$.

The following theorem is useful for computing examples.

\begin{Theorem}\label{Theorem:2.5}
Assume $\nu$ is a Nakayama functor relative to $f_!\dashv f^*\colon \cD\to \cA$. The following holds:
\begin{enumerate}
\item\label{Theorem:2.5:1} $A\in \relGproj{P}{\cA}$ if and only if 
\begin{enumerate}
\item $L_i\nu(A)=0$ for all $i>0$;
\item $R^i\nu^-(\nu(A))=0$ for all $i>0$;
\item $\lambda_A\colon A\to \nu^-\nu(A)$ is an isomorphism.
\end{enumerate}
\item\label{Theorem:2.5:2} If $P$ is Iwanaga-Gorenstein, then 
\[
\relGproj{P}{\cA}=\{A\in \cA \mid  L_i\nu(A)=0 \text{ for all }i>0\}.
\]
\end{enumerate}
\end{Theorem}

\begin{proof}
This is \cite[Proposition 4.1.3]{Kva17} and \cite[Theorem 4.2.2]{Kva17}.
\end{proof}

\begin{Example}[Example 3.2.6 in \cite{Kva17}]\label{Example:3}
Let $k$ be a commutative ring, and let $\Lambda_1$ and $\Lambda_2$ be $k$-algebras. Consider the adjoint pair $f_!\dashv f^*$ where $f^*$ is the restriction functor
\[
f^*:= \res^{\Lambda_1\otimes_k \Lambda_2}_{\Lambda_2}\colon (\Lambda_1\otimes_k \Lambda_2) \text{-}\Md \to \Lambda_2 \text{-}\Md
\] 
and $f_!:=\Lambda_1\otimes_k -\colon \Lambda_2 \text{-}\Md \to (\Lambda_1\otimes_k \Lambda_2) \text{-}\Md$. If $\Lambda_1$ is finitely generated projective as a $k$-module, then the functor
\[
\nu:=\Hom_k(\Lambda_1,k)\otimes_{\Lambda_1} -\colon (\Lambda_1\otimes_k \Lambda_2)\text{-} \Md \to (\Lambda_1\otimes_k \Lambda_2) \text{-} \Md.  
\]
is a Nakayama functor relative to $f_!\dashv f^*$. 
\end{Example}

\begin{Example}\label{finitely presented}
Assume $k$, $\Lambda_1$ and $\Lambda_2$ are as in Example \ref{Example:3}. If in addition $\Lambda_2$ is left coherent, then the categories $\Lambda_2 \text{-}\md$ and $(\Lambda_1\otimes_k \Lambda_2)\text{-}\md$ of finitely presented left modules are abelian. In this case $f^*$, $f_!$ and $\nu$ restrict to functors 
\begin{align*}
& f^*:=\res^{\Lambda_1\otimes_k \Lambda_2}_{\Lambda_2}\colon (\Lambda_1\otimes_k \Lambda_2) \text{-}\md \to \Lambda_2 \text{-}\md \\
& f_!:=\Lambda_1\otimes_k -\colon \Lambda_2 \text{-}\md \to (\Lambda_1\otimes_k \Lambda_2) \text{-}\md \\
& \nu:=\Hom_k(\Lambda_1,k)\otimes_{\Lambda_1} -\colon (\Lambda_1\otimes_k \Lambda_2)\text{-} \md \to (\Lambda_1\otimes_k \Lambda_2) \text{-} \md.
\end{align*} 
and $\nu$ is still a Nakayama functor relative to $f_!\dashv f^*$ in this case. 
\end{Example}

\section{Lifting Frobenius exact subcategories}\label{Lifting Frobenius exact subcategories}

In this section we fix abelian categories $\cA$ and $\cB$, a faithful functor $f^*\colon \cA\to \cB$ with left adjoint $f_!\colon \cB\to \cA$, and we assume $f_!\dashv f^*$ has a Nakayama functor $\nu\colon \cA\to \cA$.  Our goal is to investigate when the subcategory $(f^*\circ \nu)^{-1}(\Gproj(\cB))\cap \relGproj{P}{\cA}$ is equal to $\Gproj(\cA)$ if $\cA$ and $\cB$ have enough projectives. In the first part we show that $(f^*\circ \nu)^{-1}(\cF)\cap \cX$ is an admissible subcategory of $\Gproj(\cA)$ if $\cX$ is a $P$-admissible subcategory of $\relGproj{P}{\cA}$ and $\cF$ is an admissible subcategory of $\Gproj(\cB)$.

\subsection{\texorpdfstring{$P$}{}-admissible subcategories of \texorpdfstring{$\relGproj{P}{\cA}$}{}}\label{Admissible subcategories}

\begin{Definition}\label{Definition:10}
A full subcategory $\cX\subset \cA$ is called a $P$-\emphbf{admissible subcategory of} $\relGproj{P}{\cA}$ if it is closed under extensions, direct summands, and satisfies the following properties:
\begin{enumerate}
  \item \label{Definition:10,1}$\cX$ contains all the $P$-projective objects of $\cA$;
	\item \label{Definition:10,2}$L_1\nu(X)=0$ for all $X\in \cX$;
  \item \label{Definition:10,3}For all $X\in \cX$ there exists a short exact sequence $0\to X'\xrightarrow{} A\xrightarrow{} X\to 0$ with $A$ being $P$-projective and $X'\in \cX$;
	\item \label{Definition:10,4}For all $X\in \cX$ there exists a short exact sequence $0\to X\xrightarrow{} A\xrightarrow{} X'\to 0$ with $A$ being $P$-projective and $X'\in \cX$. 
\end{enumerate}
\end{Definition}

The following result is immediate from the definition. 

\begin{Proposition}\label{Proposition:2*}
The following hold:
\begin{enumerate}
	\item\label{Proposition:2,1*} $\relGproj{P}{\cA}$ is a $P$-admissible subcategory of $\relGproj{P}{\cA}$;
	\item\label{Proposition:2,2*} Assume $\cX$ is a $P$-admissible subcategory of $\relGproj{P}{\cA}$. Then $\cX\subset \relGproj{P}{\cA}$.
\end{enumerate} 
\end{Proposition}

\begin{Example}\label{Example:4}
Let $\Lambda$ be a finite-dimensional algebra over a field $k$. Furthermore, let $g^*\colon \Lambda\text{-}\Md\to k\text{-}\Md$ be the restriction functor and $g_!=\Lambda \otimes_k -\colon k\text{-}\Md\to \Lambda\text{-}\Md$ its left adjoint. As stated in Example \ref{Example:3}, the adjoint pair $g_!\dashv g^*$ has Nakayama functor 
\[
\nu'=\Hom_k(\Lambda,k)\otimes_{\Lambda}-\colon \Lambda\text{-}\Md\to \Lambda\text{-}\Md
\]
In this case the $P'$-projective objects are just the projective $\Lambda$-modules, where $P':= g_!\circ g^*$. Also, $L_1\nu'(M)= \Tor^{\Lambda}_1(\Hom_k(\Lambda,k),M)=0$ if and only if 
\[
\Hom_k(\Tor^{\Lambda}_1(\Hom_k(\Lambda,k),M),k)\cong \Ext^1_{\Lambda}(M,\Lambda )=0.
\]
In this case $\Ext^1_{\Lambda}(M,\prod\Lambda )\cong \prod \Ext^1_{\Lambda}(M,\Lambda ) = 0$.  Since any projective $\Lambda$-module is a direct summand of a product $\prod \Lambda$ when $\Lambda$ is finite-dimensional, it follows that $L_1\nu'(M)=0$ if and only if $\Ext^1_{\Lambda}(M,Q)=0$ for any $Q\in \Proj (\Lambda\text{-}\Md)$. Hence, the $P'$-admissible subcategories of $\relGproj{P'}{\Lambda\text{-}\Md}$ are precisely the admissible subcategories of $\Gproj(\Lambda\text{-}\Md)$. In particular, it follows that
\[
\Gproj(\Lambda\text{-}\Md)=\relGproj{P'}{\Lambda\text{-}\Md}.
\]
\end{Example} 

In the following we consider the adjunctions 
\begin{align*}
& \adjiso{f^*\circ \nu}{f_!}\colon \cB(f^*\nu(A),B)\xrightarrow{\cong}\cA(A,f_!(B)) \\ & \adjiso{f_!}{f^*}\colon \cA(f_!(B),A)\xrightarrow{\cong}\cB(B,f^*(A))
\end{align*} 
with units and counits
\begin{align*}
& \unit{f^*\circ \nu}{f_!}\colon 1_{\cA}\to f_!\circ f^*\circ \nu \quad \counit{f^*\circ \nu}{f_!}\colon f^*\circ \nu\circ f_! \to 1_{\cB} \\
& \unit{f_!}{f^*}\colon 1_{\cB}\to f^*\circ f_! \quad \counit{f_!}{f^*}\colon f_!\circ f^* \to 1_{\cA}
\end{align*} 
Since $f^*$ is faithful, it follows that $\counit{f_!}{f^*}$ is an epimorphism. 

\begin{Lemma}\label{Lemma:6}
Let $\cX$ be a $P$-admissible subcategory of $\relGproj{P}{\cA}$, and let $X\in \cX$. The following holds:
\begin{enumerate}
	\item\label{Lemma:6,1} $\unit{f^*\circ \nu}{f_!}_X$ is a monomorphism and $\Coker \unit{f^*\circ \nu}{f_!}_X\in \cX$;
	\item\label{Lemma:6,2}$\Ker \counit{f_!}{f^*}_X\in \cX$.
\end{enumerate} 
\end{Lemma}

\begin{proof}
We prove \ref{Lemma:6,1}. Since $X\in \cX$ there exists an exact sequence $0\to X\xrightarrow{i} f_!(B)\xrightarrow{} X'\to 0$ with $X'\in \cX$ and $B\in \cB$. Since $i=f_!((\adjiso{f^*\circ \nu}{f_!})^{-1}(i))\circ \unit{f^*\circ \nu}{f_!}_X$, it follows that $\unit{f^*\circ \nu}{f_!}_X$ is a monomorphism. We therefore have a commutative diagram
\begin{equation*}
\begin{tikzpicture}[description/.style={fill=white,inner sep=2pt}]
\matrix(m) [matrix of math nodes,row sep=3.5em,column sep=5.0em,text height=1.5ex, text depth=0.25ex] 
{ X & f_!f^*\nu(X) & \Coker \unit{f^*\circ \nu}{f_!}_X  \\
  X & f_!(B) & X'  \\};
\path[->]
(m-1-1) edge node[auto] {$\unit{f^*\circ \nu}{f_!}_X$} 	    				    (m-1-2)
(m-1-2) edge node[auto] {$$} 	    											(m-1-3)
(m-2-1) edge node[auto] {$i$} 	    										    (m-2-2)
(m-2-2) edge node[auto] {$$} 	    										  	(m-2-3)

(m-1-1) edge node[auto] {$1_X$} 	    												    (m-2-1)
(m-1-2) edge node[auto] {$f_!((\adjiso{f^*\circ \nu}{f_!})^{-1}(i))$} 	    						(m-2-2)
(m-1-3) edge node[auto] {$$} 	    								    					(m-2-3);	
\end{tikzpicture}
\end{equation*}
where the rows are short exact sequences. By the dual of Lemma 5.2 in \cite{Pop73} it follows that the right square is a pushforward and a pullback square. Hence we get a short exact sequence
\[
0\to f_!f^*\nu(X)\to f_!(B)\oplus \Coker \unit{f^*\circ \nu}{f_!}_X \to X'\to 0.
\] 
Since $\cX$ is closed under extensions and direct summands, it follows that $\Coker \unit{f^*\circ \nu}{f_!}_X\in \cX$. 

For \ref{Lemma:6,2}, choose an exact sequence $0\to X''\xrightarrow{} f_!(B')\xrightarrow{p} X\to 0$ with $X''\in \cX$ and $B'\in \cB$. We then get a commutative diagram

\begin{equation*}
\begin{tikzpicture}[description/.style={fill=white,inner sep=2pt}]
\matrix(m) [matrix of math nodes,row sep=3.5em,column sep=5.0em,text height=1.5ex, text depth=0.25ex] 
{ \Ker \counit{f_!}{f^*}_X & f_!f^*(X) & X  \\
  X'' & f_!(B') & X  \\};
\path[->]
(m-1-1) edge node[auto] {$$} 	    													    (m-1-2)
(m-1-2) edge node[auto] {$\counit{f_!}{f^*}_X$} 	    										(m-1-3)
(m-2-1) edge node[auto] {$$} 	    													    (m-2-2)
(m-2-2) edge node[auto] {$p$} 	    										  			  (m-2-3)

(m-2-1) edge node[auto] {$$} 	    												      (m-1-1)
(m-2-2) edge node[auto] {$f_!(\adjiso{f_!}{f^*}(p))$} 	    							  (m-1-2)
(m-2-3) edge node[auto] {$1_X$} 	    								    			(m-1-3);	
\end{tikzpicture}
\end{equation*}
where the rows are short exact sequences. The left square is a pushforward and a pullback square, and therefore gives rise to an exact sequence
\[
0\to X''\to f_!(B')\oplus \Ker \counit{f_!}{f^*}_X \to f_!f^*(X)\to 0.
\]
Since $\cX$ is closed under extensions and direct summands, it follows that $\Ker \counit{f_!}{f^*}_X\in \cX$.
\end{proof}

\begin{Lemma}\label{Lemma:10}
Let $\cX$ be a $P$-admissible subcategory of $\relGproj{P}{\cA}$. The following holds:
\begin{enumerate}
\item\label{Lemma:10:1} Let $s\colon X\to f_!(B)$ be a morphism in $\cA$ with $X\in \cX$ and $B\in \cB$. Assume $(\adjiso{f^*\circ \nu}{f_!})^{-1}(s)\colon f^*\nu(X)\to B$ is a monomorphism. Then $s$ is a monomorphism and $\Coker s\in \cX$;
\item\label{Lemma:10:2} Let $s'\colon f_!(B)\to X$ be a morphism in $\cA$ with $X\in \cX$ and $B\in \cB$. Assume that $\adjiso{f_!}{f^*}(s')\colon B\to f^*(X)$ is an epimorphism. Then $s'$ is an epimorphism and $\Ker s'\in \cX$.
\end{enumerate}
\end{Lemma}

\begin{proof}
We only prove part \ref{Lemma:10:1}, part \ref{Lemma:10:2} is proved dually. Consider the commutative diagram
\[
\begin{tikzpicture}[description/.style={fill=white,inner sep=2pt}]
\matrix(m) [matrix of math nodes,row sep=3.5em,column sep=5.0em,text height=1.5ex, text depth=0.25ex] 
{ X & f_!f^*\nu(X) & \Coker \unit{f^*\circ \nu}{f_!}_X  \\
   X &  f_!(B)   & \Coker s       \\};
\path[->]
(m-1-1) edge node[auto] {$\unit{f^*\circ \nu}{f_!}_X$} 	    						(m-1-2)
(m-1-2) edge node[auto] {$$} 	    												(m-1-3)
(m-2-1) edge node[auto] {$s$} 	    												(m-2-2)
(m-2-2) edge node[auto] {$$} 	    												(m-2-3)

(m-1-1) edge node[auto] {$1_X$} 	    								   					  (m-2-1)
(m-1-2) edge node[auto] {$f_!((\adjiso{f^*\circ \nu}{f_!})^{-1}(s))$} 	    									  (m-2-2)
(m-1-3) edge node[auto] {$t$} 	    								 						  	(m-2-3);	
\end{tikzpicture}
\]
where $t$ is induced from the commutativity of the left square. Since $\unit{f^*\circ \nu}{f_!}_X$ is a monomorphism by Lemma \ref{Lemma:6}, we get that $s$ is a monomorphism. Hence, the upper and lower row are short exact sequences. Therefore, by the snake lemma $t$ is a monomorphism and
\[
\Coker t \cong \Coker f_!((\adjiso{f^*\circ \nu}{f_!})^{-1}(s)) \cong f_! (\Coker (\adjiso{f^*\circ \nu}{f_!})^{-1}(s))
\]
Hence, we get an exact sequence
\[
0\to \Coker \unit{f^*\circ \nu}{f_!}_X\xrightarrow{t} \Coker s\to f_! (\Coker (\adjiso{f^*\circ \nu}{f_!})^{-1}(s))\to 0. 
\]
Since $\cX$ is closed under extensions, $f_! (\Coker (\adjiso{f^*\circ \nu}{f_!})^{-1}(s))$ is $P$-projective, and $\Coker \unit{f^*\circ \nu}{f_!}_X\in \cX$ by Lemma \ref{Lemma:6}, we get that $\Coker s\in \cX$. 
\end{proof}

\begin{Example}\label{Example:5}
Let $k$ be a field, let $\Lambda_1$ be a finite-dimensional $k$-algebra, and let $\Lambda_2$ be a $k$-algebra which is left coherent. Let $f_!\dashv f^*$ be the adjoint pair with Nakayama functor $\nu$ as in Example \ref{finitely presented}. Let $\cF\subset \Gproj(\Lambda_1\text{-}\Md)$ be an admissible subcategory. We claim that the category 
\[
\cX =\{ M\in (\Lambda_1\otimes_k \Lambda_2)\text{-}\md \mid\text{ } {}_{\Lambda_1}M\in \cF\}
\]
is a $P$-admissible subcategory of $\relGproj{P}{(\Lambda_1\otimes_k \Lambda_2)\text{-}\md}$, where $P:=f_!\circ f^*$: Indeed, the $P$-projective objects are summands of modules of the form $\Lambda_1\otimes_k M$. Since they are projective when restricted to $\Lambda_1\text{-}\Md$, they are contained in $\cX$, which shows \ref{Definition:10,1}. Furthermore, for $M\in \cX$ we have $L_1\nu(M)= \Tor^{\Lambda_1}_1(\Hom_k(\Lambda_1,k),M)$, and this is $0$ since ${}_{\Lambda_1}M\in \cF \subset \Gproj(\Lambda_1\text{-}\Md)$ and $\Hom_k(\Tor^{\Lambda_1}_1(\Hom_k(\Lambda_1,k),M),k)\cong \Ext^1_{\Lambda_1}(M,\Lambda_1)$. This shows \ref{Definition:10,2}. Also, $\cX$ is closed under kernels of epimorphisms by Lemma \ref{Lemma:1}, and hence it satisfies \ref{Definition:10,3}. It only remains to show \ref{Definition:10,4}: By Example \ref{Example:4} we know that $\cF$ is a $P'$-admissible subcategory of $\relGproj{P'}{\Lambda_1\text{-}\Md}$, where $P'=g_!\circ g^*$ and $g^*\colon \Lambda_1\text{-}\Md\to k\text{-}\Md$ is the restriction with left adjoint $g_!=\Lambda_1\otimes_k-\colon k\text{-}\Md\to \Lambda_1\text{-}\Md$. Consider the exact sequence
\[
0\to M\xrightarrow{\unit{f^*\circ \nu}{f_!}_M} f_!f^*\nu (M) \to \Coker \unit{f^*\circ \nu}{f_!}_M \to 0
\] 
of $\Lambda_1\otimes_k \Lambda_2$-modules, Restricting to $\Lambda_1\text{-}\Md$ gives the exact sequence 
\[
0\to {}_{\Lambda_1}M\xrightarrow{\unit{g^*\circ \nu'}{g_!}_{{}_{\Lambda_1}M}} g_!g^*\nu' ({}_{\Lambda_1}M) \to \Coker \unit{g^*\circ \nu'}{g_!}_{{}_{\Lambda_1}M} \to 0
\]
 It follows from Lemma \ref{Lemma:6} that $\Coker \unit{g^*\circ \nu'}{g_!}_{{}_{\Lambda_1}M}\in \cF$. Therefore, we have that $\Coker \unit{f^*\circ \nu}{f_!}_M\in \cX$. This implies that $\cX$ satisfies \ref{Definition:10,4}, which proves the claim. In particular, $\cX$ is a $P$-admissible subcategory of $\relGproj{P}{(\Lambda_1\otimes_k \Lambda_2)\text{-}\md}$ when $\cF =\Gproj (\Lambda_1 \text{-}\Md)$ or $\cF= \Proj (\Lambda_1 \text{-}\Md  )$.

Now assume $\cF =\Gproj (\Lambda_1 \text{-}\Md)$. We claim that $\cX=\relGproj{P}{(\Lambda_1\otimes_k \Lambda_2)\text{-}\md}$. By the argument above we know that $\cX\subset \relGproj{P}{(\Lambda_1\otimes_k \Lambda_2)\text{-}\md}$, so we only need to show the other inclusion. Assume $M\in \relGproj{P}{(\Lambda_1\otimes_k \Lambda_2)\text{-}\md}$, and let $A_{\bullet}$ be an exact sequence in $(\Lambda_1\otimes_k \Lambda_2)\text{-}\md$ with $Z_0(A_{\bullet})=M$ as in Definition \ref{Definition:7}. Note that the components of ${}_{\Lambda_1}A_{\bullet}$ are projective $\Lambda_1$-modules. Furthermore, since the sequence $\nu(A_{\bullet})$ is exact, the sequence 
\[
\Hom_k(\nu(A_{\bullet}),k)=\Hom_k(\Hom_k(\Lambda_1,k)\otimes_{\Lambda_1}A_{\bullet},k) \cong \Hom_{\Lambda_1}(A_{\bullet},\Lambda_1)
\]
is exact. Since any projective $\Lambda_1$-module is a summand of a product of $\Lambda_1$, and $\Hom_{\Lambda_1}(A_{\bullet},\prod \Lambda_1)\cong \prod \Hom_{\Lambda_1}(A_{\bullet},\Lambda_1)$ is exact, it follows that ${}_{\Lambda_1}A_{\bullet}$ is a totally acyclic complex of $\Lambda_1$-modules. This shows that ${}_{\Lambda_1}M\in \Gproj(\Lambda_1\text{-}\Md)$, and the claim follows.
\end{Example} 

\begin{Example}\label{Example:6}
Let $k$ be a field, let $\Lambda_1$ be a finite-dimensional $k$-algebra, and let $\Lambda_2$ be a $k$-algebra. Let $f_!\dashv f^*$ be the adjoint pair with Nakayama functor $\nu$ as in Example \ref{Example:3}. By a similar argument as in Example \ref{Example:5} we get that if $\cF\subset \Lambda_1\text{-}\Md$ is an admissible subcategory of $\Gproj(\Lambda_1\text{-}\Md)$, then 
\[
\cX =\{ M\in (\Lambda_1\otimes_k \Lambda_2)\text{-}\Md \mid \text{ } {}_{\Lambda_1}M\in \cF\}
\]
is a $P$-admissible subcategory of $\relGproj{P}{(\Lambda_1\otimes_k \Lambda_2)\text{-}\Md}$, where $P=f_!\circ f^*$. Also, we get that
\[
\relGproj{P}{(\Lambda_1\otimes_k\Lambda_2)\text{-}\Md) =\{ M\in (\Lambda_1\otimes_k \Lambda_2)\text{-}\Md \mid \text{ } {}_{\Lambda_1}M\in \Gproj(\Lambda_1\text{-}\Md}\}.
\]

\end{Example}

\subsection{Lifting admissible subcategories}\label{Lifting admissible subcategories}

Note that $f_!$ preserves projective objects since it has an exact right adjoint. In fact, we have the following result.

\begin{Lemma}
Assume $\cB$ has enough projectives. Then the full subcategory 
\[
f_!(\Proj(\cB)):= \{f_!(Q)\mid Q\in \Proj(\cB)\}
\] 
is generating in $\cA$. In particular, $\cA$ has enough projectives.
\end{Lemma}

\begin{proof}
For $A\in \cA$ choose an epimorphism $Q\xrightarrow{p} f^*(A)$ in $\cB$ with $Q$ projective. The composition $f_!(Q)\xrightarrow{f_!(p)} f_! f^*(A)\xrightarrow{\counit{f_!}{f^*}_A} A$ is then an epimorphism in $\cA$. This proves the claim.
\end{proof}

For the remainder of this section we assume $\cB$ has enough projective objects. Furthermore, we fix a $P$-admissible subcategory $\cX$ of $\relGproj{P}{\cA}$ and an admissible subcategory $\cF$ of $\Gproj(\cB)$. Let 
\[
(f^*\circ \nu)^{-1}(\cF):= \{ A\in \cA \mid f^*\nu(A)\in \cF\}
\] 
Our goal is to show that $(f^*\circ \nu)^{-1}(\cF)\cap \cX$ is an admissible subcategory of $\Gproj(\cA)$.

\begin{Lemma}\label{Lemma:12}
The category $(f^*\circ \nu)^{-1}(\cF)\cap \cX$ is closed under extensions and direct summands in $\cA$.
\end{Lemma}

\begin{proof}
It is immediate that $(f^*\circ \nu)^{-1}(\cF)\cap \cX$ is closed under direct summands. We show that it is closed under extensions. Let $0\to A_1\xrightarrow{s} A_2\xrightarrow{t} A_3\to 0$ be an exact sequence in $\cA$ with $A_1,A_3\in (f^*\circ \nu)^{-1}(\cF)\cap \cX$. Since $\cX$ is closed under extensions, it follows that $A_2\in \cX$. Also, since $L_1\nu(A_3)=0$, we have an exact sequence
\[
0\to f^*\nu(A_1)\xrightarrow{f^*\nu(s)} f^*\nu(A_2)\xrightarrow{f^*\nu(t)} f^*\nu(A_3)\to 0
\]
in $\cB$. Since $\cF$ is closed under extensions, it follows that $f^*\nu(A_2)\in \cF$. This proves the claim.
\end{proof}

Since $(f^*\circ \nu)^{-1}(\cF)\cap \cX$ is closed under extensions, it inherits an exact structure from $\cA$.

\begin{Lemma}\label{Lemma:13}
The category $(f^*\circ \nu)^{-1}(\cF)\cap \cX$ contains the projective objects in $\cA$.
\end{Lemma}

\begin{proof}
Let $Q\in \cB$ be projective. Then $f^*\nu f_! (Q)$ is projective since the functor $\cB(f^*\nu f_! (Q),-)\cong \cB(Q,f^*f_!(-))$ is exact. Since $\cX$ contains all the $P$-projective objects of $\cA$ and $\cF$ contains all the projective objects of $\cB$, it follows that $f_!(Q)\in (f^*\circ \nu)^{-1}(\cF)\cap \cX$. Since any projective object in $\cA$ is a summand of an object of the form $f_!(Q)$, the claim follows. 
\end{proof}

\begin{Lemma}\label{Lemma:15}
We have $\Ext^{i}_{\cA}(A,Q)=0$ for all $i>0$, $A\in (f^*\circ \nu)^{-1}(\cF)\cap \cX$, and $Q\in \cA$ projective. 
\end{Lemma}

\begin{proof}
We only need to show the statement for $Q=f_!(Q')$ where $Q'\in \cB$ is projective. Note first that any exact sequence $0\to f_!(Q')\to \cdots \to A\to 0$ stays exact under the functor $f^*\circ \nu$ since $L_i\nu(A)=0$ for all $i>0$ and as $f^*$ is exact. Since we have an adjunction $f^*\circ \nu\dashv f_!$ and the functor $f_!$ is exact it follows that $\Ext^i_{\cA}(A,f_!(Q'))\cong  \Ext^i_{\cB}(f^*\nu(A),Q')$ by Lemma 6.1 in \cite{HJ16}.  Since the latter is $0$ by the assumption on $A$, the claim follows.
\end{proof}

\begin{Lemma}\label{Lemma:14}
If $A\in (f^*\circ \nu)^{-1}(\cF)\cap \cX$, then there exists a projective object $Q\in \cA$ and an epimorphism $p\colon Q\to A$ such that $\Ker p\in (f^*\circ \nu)^{-1}(\cF)\cap \cX$.  
\end{Lemma}

\begin{proof}
Let $A\in (f^*\circ \nu)^{-1}(\cF)\cap \cX$ be arbitrary, and choose an epimorphism $q\colon Q'\to f^*(A)$ in $\cB$ with $Q'$ projective. By Lemma \ref{Lemma:10} part \ref{Lemma:10:2} the morphism $(\adjiso{f_!}{f^*})^{-1}(q)\colon f_!(Q')\to A$ is an epimorphism and $\Ker (\adjiso{f_!}{f^*})^{-1}(q)\in \cX$. Since $f_!(Q')$ is projective, it only remains to show $\Ker (\adjiso{f_!}{f^*})^{-1}(q) \in (f^*\circ \nu)^{-1}(\cF)$. To this end, note that applying $f^*\circ \nu$ to 
\[
0\to \Ker (\adjiso{f_!}{f^*})^{-1}(q)\to f_!(Q')\xrightarrow{(\adjiso{f_!}{f^*})^{-1}(q)}A\to 0
\]
 gives an exact sequence
\[
0\to f^*\nu(\Ker (\adjiso{f_!}{f^*})^{-1}(q))\to f^*\nu f_!(Q')\xrightarrow{f^*\nu((\adjiso{f_!}{f^*})^{-1}(q))}f^*\nu(A)\to 0
\]
in $\cB$ since $L_1\nu(A)=0$. By Lemma \ref{Lemma:1} we have that $\cF$ is resolving, and therefore $f^*\nu(\Ker (\adjiso{f_!}{f^*})^{-1}(q))\in \cF$. This proves the claim.
\end{proof}

\begin{Lemma}\label{Lemma:16}
If $A\in (f^*\circ \nu)^{-1}(\cF)\cap \cX$, then there exists a projective object $Q\in \cA$ and a monomorphism $j\colon A\to Q$ such that $\Coker j\in (f^*\circ \nu)^{-1}(\cF)\cap \cX$. 
\end{Lemma}

\begin{proof}
Let $A\in (f^*\circ \nu)^{-1}(\cF)\cap \cX$ be arbitrary. Choose a projective object $Q'\in \cB$ and an exact sequence
\[
0\to f^*\nu(A)\xrightarrow{i} Q' \xrightarrow{p} B\to 0
\]
with $B\in \cF$. By Lemma \ref{Lemma:10} we get that $j:=\adjiso{f^*\circ \nu}{f_!}(i)\colon A\to f_!(Q')$ is a monomorphism and $\Coker j\in \cX$. Since $f_!(Q')$ is projective, it only remains to show that $\Coker j\in (f^*\circ \nu)^{-1}(\cF)$. To this end, note that we have a commutative diagram 
\[
\begin{tikzpicture}[description/.style={fill=white,inner sep=2pt}]
\matrix(m) [matrix of math nodes,row sep=2.5em,column sep=4em,text height=1.5ex, text depth=0.25ex] 
{ f^*\nu(A) & f^*\nu f_!(Q') & f^*\nu(\Coker j) \\
  f^*\nu (A) & Q' & B  \\};
\path[->]
(m-1-1) edge node[auto] {$f^*\nu(j)$} 	    								(m-1-2)
(m-1-2) edge node[auto] {$$} 	    										(m-1-3)

(m-2-1) edge node[auto] {$i$} 	    													    (m-2-2)
(m-2-2) edge node[auto] {$$} 	    													    (m-2-3)

(m-1-1) edge node[auto] {$1_{f^*\nu (A)}$} 	    								   				(m-2-1)
(m-1-2) edge node[auto] {$\counit{f^*\circ \nu}{f_!}_{Q'}$} 	    									  				(m-2-2)
(m-1-3) edge node[auto] {$$} 	    								 					   (m-2-3);	
\end{tikzpicture}
\]
where the rows are short exact sequences. Hence, the right square is a pullback and a pushout square. Therefore, we get an exact sequence
\[
0\to f^*\nu f_! (Q')\to f^*\nu (\Coker j) \oplus Q'\to B\to 0.
\]
We know that $B\in \cF$, $f^*\nu f_!(Q')$ is projective, and $\cF$ is closed under extensions and direct summands. Therefore, it follows that $f^*\nu(\Coker j)\in \cF$. This proves the claim.
\end{proof}

\begin{Theorem}\label{Theorem:3}
The category $(f^*\circ \nu)^{-1}(\cF)\cap \cX$ is an admissible subcategory of $\Gproj(\cA)$. 
\end{Theorem}

\begin{proof}
This follows from Lemma \ref{Lemma:12}, \ref{Lemma:13}, \ref{Lemma:15}, \ref{Lemma:14} and \ref{Lemma:16}. 
\end{proof}

\begin{Example}\label{Example:7}
Let $k$ be a field, let $\Lambda_1$ be a finite-dimensional algebra over $k$, and let $\Lambda_2$ be a left coherent $k$-algebra. Theorem \ref{Theorem:3} together with Example \ref{Example:5} show that the categories
\begin{enumerate}
	\item $\{M\in (\Lambda_1\otimes_k\Lambda_2)\text{-}\md \mid \text{ } _{\Lambda_1}M\in\Gproj(\Lambda_1\text{-}\Md ) \\ \text{ and }  _{\Lambda_2}(\Hom_k(\Lambda_1,k)\otimes_{\Lambda_1} M)\in \Gproj(\Lambda_2\text{-}\md )\}$
	\item $\{ M\in (\Lambda_1\otimes_k\Lambda_2)\text{-}\md \mid \text{ } _{\Lambda_1}M\in\Gproj(\Lambda_1\text{-}\Md ) \\ \text{ and }  _{\Lambda_2}(\Hom_k(\Lambda_1,k)\otimes_{\Lambda_1} M)\in \Proj(\Lambda_2\text{-}\md )\}$
	\item $\{ M\in (\Lambda_1\otimes_k\Lambda_2)\text{-}\md \mid \text{ } _{\Lambda_1}M\in\Proj(\Lambda_1\text{-}\Md ) \\ \text{ and }  _{\Lambda_2}(\Hom_k(\Lambda_1,k)\otimes_{\Lambda_1} M)\in \Gproj(\Lambda_2\text{-}\md )\}$
	\item $\{ M\in (\Lambda_1\otimes_k\Lambda_2)\text{-}\md \mid \text{ } _{\Lambda_1}M\in\Proj(\Lambda_1\text{-}\Md ) \\ \text{ and }   _{\Lambda_2}(\Hom_k(\Lambda_1,k)\otimes_{\Lambda_1} M)\in \Proj(\Lambda_2\text{-}\md )\}$
\end{enumerate}
are admissible subcategories of $\Gproj((\Lambda_1\otimes_k \Lambda_2)\text{-}\md)$.
\end{Example}

\begin{Example}\label{Example:8}
Let $k$ be a field, let $\Lambda_1$ be a finite-dimensional algebra over $k$, and let $\Lambda_2$ be a $k$-algebra. Example \ref{Example:6} together with Theorem \ref{Theorem:3} show that the categories 
\begin{enumerate}
	\item $\{M\in (\Lambda_1\otimes_k\Lambda_2)\text{-}\Md \mid \text{ } _{\Lambda_1}M\in\Gproj(\Lambda_1\text{-}\Md ) \\ \text{ and }  _{\Lambda_2}(\Hom_k(\Lambda_1,k)\otimes_{\Lambda_1} M)\in \Gproj(\Lambda_2\text{-}\Md )\}$
	\item $\{ M\in (\Lambda_1\otimes_k\Lambda_2)\text{-}\Md \mid\text{ } _{\Lambda_1}M\in\Gproj(\Lambda_1\text{-}\Md ) \\ \text{ and }  _{\Lambda_2}(\Hom_k(\Lambda_1,k)\otimes_{\Lambda_1} M)\in \Proj(\Lambda_2\text{-}\Md )\}$
	\item $\{ M\in (\Lambda_1\otimes_k\Lambda_2)\text{-}\Md \mid\text{ } _{\Lambda_1}M\in\Proj(\Lambda_1\text{-}\Md ) \\ \text{ and }  _{\Lambda_2}(\Hom_k(\Lambda_1,k)\otimes_{\Lambda_1} M)\in \Gproj(\Lambda_2\text{-}\Md )\}$
	\item $\{ M\in (\Lambda_1\otimes_k\Lambda_2)\text{-}\Md \mid\text{ } _{\Lambda_1}M\in\Proj(\Lambda_1\text{-}\Md ) \\ \text{ and }   _{\Lambda_2}(\Hom_k(\Lambda_1,k)\otimes_{\Lambda_1} M)\in \Proj(\Lambda_2\text{-}\Md )\}$
\end{enumerate}
are admissible subcategories of $\Gproj((\Lambda_1\otimes_k\Lambda_2)\text{-}\Md)$.
\end{Example}

\subsection{Lifting Gorenstein projectives}\label{Lifting Gorenstein projectives}

Now assume $\cX=\relGproj{P}{\cA}$ and $\cF=\Gproj(\cB)$. We define
\[
\Gproj(\relGproj{P}{\cA}):=(f^*\circ \nu)^{-1}(\Gproj(\cB))\cap \relGproj{P}{\cA}.
\]
 By Theorem \ref{Theorem:3} we know that $\Gproj(\relGproj{P}{\cA})$ is an admissible subcategory of $\Gproj(\cA)$, and therefore
\[
\Gproj(\relGproj{P}{\cA})\subset \Gproj(\cA).
\]
We want to investigate when this inclusion is an equality. We first give a different description of the objects in $\Gproj(\relGproj{P}{\cA})$.

\begin{Proposition}\label{Proposition:3}
Let $A\in \cA$ be arbitrary. Then $A\in \Gproj(\relGproj{P}{\cA})$ if and only if there exists a totally acyclic complex 
\[
Q_{\bullet} = \cdots \xrightarrow{s_{2}} Q_{1}\xrightarrow{s_{1}} Q_0\xrightarrow{s_0} Q_{-1}\xrightarrow{s_{-1}} \cdots
\]
in $\cA$, such that $Z_i(Q_{\bullet})\in \relGproj{P}{\cA}$ for all $i\in \bZ$, and such that $Z_0(Q_{\bullet})=A$.
\end{Proposition}

\begin{proof}
Assume $A\in \Gproj(\relGproj{P}{\cA})$. Since $\Gproj(\relGproj{P}{\cA})$ is an admissible subcategory of $\Gproj(\cA)$, we can find a long exact sequence
\[
Q_{\bullet} = \cdots \to Q_{1}\to Q_0\to Q_{-1}\to \cdots
\]
with $Q_i\in \cA$ projective, $Z_0(Q_{\bullet})=A$, and $Z_i(Q_{\bullet})\in \Gproj(\relGproj{P}{\cA})$ for all $i\in \bZ$. Furthermore, $\Ext^1_{\cA}(A',Q')=0$ for all $A'\in \Gproj(\relGproj{P}{\cA})$ and $Q'\in \Proj(\cA)$ since $\Gproj(\relGproj{P}{\cA})$ is admissible. This shows that $Q_{\bullet}$ is totally acyclic. 

For the converse, assume $Q_{\bullet}$ is totally acyclic, $Z_i(Q_{\bullet})\in \relGproj{P}{\cA}$ for all $i\in \bZ$, and $A=Z_0(Q_{\bullet})$. The sequence 
\[
f^*\nu(Q_{\bullet})=\cdots \xrightarrow{f^*\nu(s_{2})}f^*\nu(Q_{1})\xrightarrow{f^*\nu(s_{1})} f^*\nu(Q_0)\xrightarrow{f^*\nu (s_0)} \cdots
\]
is then exact since $L_1\nu(A')=0$ for all $A'\in \relGproj{P}{\cA}$. Furthermore, the objects $f^*\nu(Q_i)\in \cB$ are projective since $f^*\circ \nu$ preserves projectives. Applying $\cB(-,Q)$ for $Q\in \cB$ projective and using the isomorphism $\cB (f^*\nu(Q_i),Q)\cong \cA(Q_i,f_!(Q))$ gives us the sequence
\[
\cdots \xrightarrow{-\circ s_{-1}}\cA(Q_{-1},f_!(Q))\xrightarrow{-\circ s_0} \cA(Q_{0},f_!(Q))\xrightarrow{-\circ s_{1}} \cA(Q_{1},f_! (Q))\xrightarrow{-\circ s_{2}} \cdots
\]
which is exact since $Q_{\bullet}$ is totally acyclic. Hence, $f^*\nu(Q_{\bullet})$ is totally acyclic, and therefore $f^*\nu (A) = Z_0(f^*\nu(Q_{\bullet}))\in \Gproj(\cB )$. This shows that $A\in \Gproj(\relGproj{P}{\cA})$, and we are done.
\end{proof}

\begin{Remark}\label{Remark:1}
Proposition \ref{Proposition:3} shows that $A\in \Gproj(\relGproj{P}{\cA})$ if and only if $A$ is Gorenstein projective inside the exact category $\relGproj{P}{\cA}$. This is the reason for the notation $\Gproj(\relGproj{P}{\cA})$.
\end{Remark}

\begin{Proposition}\label{Proposition:4}
The following statements are equivalent:
\begin{enumerate}
	\item\label{Proposition:4,1} $\Gproj(\relGproj{P}{\cA}) = \Gproj(\cA)$;  
	\item\label{Proposition:4,2} $\Gproj(\cA)\subset \relGproj{P}{\cA}$;  
	\item\label{Proposition:4,3} $f^*\circ \nu\colon \cA\to \cB$ preserves Gorenstein projectives.
\end{enumerate}
\end{Proposition}

\begin{proof}
Obviously, \ref{Proposition:4,1} $\implies$ \ref{Proposition:4,2} and \ref{Proposition:4,1}$\implies$ \ref{Proposition:4,3}. Also, if \ref{Proposition:4,2} holds then any totally acyclic complex satisfies the assumptions in Proposition \ref{Proposition:3}, and therefore \ref{Proposition:4,1} holds. We show the implication \ref{Proposition:4,3}$\implies$ \ref{Proposition:4,1}. Assume $f^*\circ \nu$ preserves Gorenstein projectives, and let $A\in \Gproj(\cA)$ be arbitrary. We only need to show that $L_1\nu(A)=0$ since this implies that if $Q_{\bullet}$ is totally acyclic, then  $\nu(Q_{\bullet})$ is exact, and hence $Z_0(Q_{\bullet})\in \relGproj{P}{\cA}$ by definition since projective objects are $P$-projective. Let 
\[
0\to A'\xrightarrow{s} Q\xrightarrow{t} A\to 0
\]
be an exact sequence in $\cA$ with $Q$ projective and $A'\in \Gproj(\cA)$. Applying $\nu$ gives an exact sequence
\[
0\to L_1\nu(A)\to \nu(A')\xrightarrow{\nu(s)} \nu (Q)\xrightarrow{\nu(t)} \nu (A)\to 0
\]
Hence, $L_1\nu(A)=0$ if and only if $\nu(s)$ is a monomorphism. Let $Q'\in \cB$ be a projective object. We know that the map $\cA(Q,f_!(Q'))\xrightarrow{-\circ s} \cA(A', f_!(Q'))$ is an epimorphism since $\Ext^1_{\cA}(A,f_!(Q'))=0$. Hence, from the adjunction $f^*\circ \nu\dashv f_!$ we get that
\[
\cB(f^*\nu(Q), Q')\xrightarrow{-\circ f^*\nu (s)} \cB(f^*\nu(A'), Q')
\]
is an epimorphism. It follows therefore from Lemma \ref{Lemma:2} that $f^*\nu(s)$ is a monomorphism. Since $f^*$ is faithful, we get that $\nu(s)$ is a monomorphism. This proves the claim.
\end{proof}

The following result gives sufficient criteria for when $\Gproj(\relGproj{P}{\cA})= \Gproj(\cA)$.

\begin{Theorem}\label{Theorem:4}
We have that $\Gproj(\relGproj{P}{\cA})= \Gproj(\cA)$ if either of the following conditions hold:
\begin{enumerate}
	\item\label{Theorem:4,1} For any long exact sequence
	\[
	0\to K\to Q_0\to Q_{-1}\to \cdots
	\]
	with $Q_i\in \cA$ projective for $i\leq 0$, we have $K\in \relGproj{P}{\cA}$;
	\item\label{Theorem:4,2} If $B\in \cB$ satisfy $\Ext^1_{\cB}(B,B')=0$ for all $B'$ of $\pdim B'< \infty$, then $B\in \Gproj(\cB)$.
\end{enumerate}
\end{Theorem}

\begin{proof}
Proposition \ref{Proposition:4} part \ref{Proposition:4,2} shows that condition \ref{Theorem:4,1} is sufficient. Assume condition \ref{Theorem:4,2} holds. By Proposition \ref{Proposition:4} part \ref{Proposition:4,3} it is sufficient to show that $f^*\nu(A)\in \Gproj(\cB)$ for all $A\in \Gproj(\cA)$. Fix $A\in \Gproj(\cA)$, and let $0\to A'\xrightarrow{s} Q\xrightarrow{t} A\to 0$ be an exact sequence in $\cA$ with $Q\in \Proj(\cA)$. Applying $f^*\circ \nu$ gives an exact sequence $ f^*\nu(A')\xrightarrow{f^*\nu(s)} f^*\nu(Q)\xrightarrow{f^*\nu(t)} f^*\nu(A)\to 0$ in $\cB$. Let $i\colon K \to f^*\nu(Q)$ be the inclusion of the kernel of $f^*\nu(t)$, let $p\colon f^*\nu(A')\to K$ be the surjection induced from $f^*\nu(s)$, and let $B'\in \cB$ be an arbitrary object. Applying $\cB(-,B)$ gives an exact sequences 
\begin{multline*}
0\to \cB(f^*\nu(A),B')\xrightarrow{-\circ f^*\nu(t)} \cB(f^*\nu(Q),B')\xrightarrow{-\circ i}  \cB(K,B') \\
\to \Ext^1_{\cB}(f^*\nu(A),B')\to 0.
\end{multline*}
where $\Ext^1_{\cB}(f^*\nu(Q),B')=0$ since $f^*\nu$ preserves projective objects. Hence, we only need to show that $-\circ i\colon \cB(f^*\nu(Q),B')\to \cB(K,B')$ is an epimorphism if $\pdim B'<\infty$. To this end, note that $\Ext^1_{\cA}(A,f_!(B))=0$ if $\pdim B'<\infty$ since $A\in \Gproj(\cA)$ and $f_!$ preserves objects of finite projective dimension. Therefore, we have an exact sequence
\[
0\to \cA(A,f_!(B'))\xrightarrow{-\circ t} \cA(Q,f_!(B'))\xrightarrow{-\circ s}  \cA(A',f_!(B'))\to 0. 
\]
 Via the adjunction $\cA(-,f_!)\cong \cB(f^*\circ \nu,-)$ the map 
 \[
 \cA(Q,f_!(B'))\xrightarrow{-\circ s}  \cA(A',f_!(B'))
 \]
  corresponds to 
  \[
  \cB(f^*\nu(Q),B')\xrightarrow{-\circ f^*\nu(s)}  \cB(f^*\nu(A'),B')
  \]
  which is therefore also an epimorphism. But $-\circ f^*\nu(s)$ factors as
 \[
 \cB(f^*\nu(Q),B')\xrightarrow{-\circ i} \cB(K,B') \xrightarrow{-\circ p}  \cB(f^*\nu(A'),B').
 \]
Since $\cB(K,B') \xrightarrow{-\circ p}  \cB(f^*\nu(A'),B')$ is a monomorphism, it follows that
\[
\cB(f^*\nu(Q),B')\xrightarrow{-\circ i} \cB(K,B')
\]
is an epimorphism. This proves the claim.
\end{proof}

\begin{Corollary}\label{Gorenstein adjoint pairs lifts Gorenstein projectives}
If $P$ is Iwanaga-Gorenstein, then 
\[
\Gproj(\relGproj{P}{\cA})= \Gproj(\cA).
\]
\end{Corollary}

\begin{proof}
This follows from condition \ref{Theorem:4,1} in Theorem \ref{Theorem:4} and the fact that $\dim_{\relGproj{P}{\cA}}(\cA)<\infty$ when $P$ is Iwanaga-Gorenstein.
\end{proof}

Recall that $\cB$ is $\Proj(\cB)$\emphbf{-Gorenstein} if $\Gpdim (B)<\infty$ for all $B\in \cB$ \cite[Corollary 4.13]{Bel00}. 

\begin{Lemma}\label{Proj Gorenstein}
If $\cB$ is $\Proj(\cB)$-Gorenstein, then 
\[
	\Gproj(\cB)= \{B\in \cB\mid  \Ext^1_{\cB}(B,B')=0  \text{ for all } B' \text{ satisfying } \pdim B'<\infty \}.  
\]
\end{Lemma} 

\begin{proof}
Assume $\Ext^1_{\cB}(B,B')=0$ for all $B'$ satisfying $\pdim B<\infty'$. Since $\Gpdim (B)<\infty$, there exists an exact sequence $0\to B_2\to B_1\to B\to 0$ such that $B_1\in \Gproj (\cB)$ and $\pdim B_2<\infty$ by \cite[Theorem 1.1]{AB89}. Since $\Ext^1_{\cB}(B,B_2)=0$ by assumption, the sequence is split. Hence, $B$ is a direct summand of $B_1$, and therefore $B\in \Gproj(\cB)$. This proves the claim.
\end{proof}

\begin{Corollary}
If $\cB$ is $\Proj(\cB)$-Gorenstein, then 
\[
\Gproj(\relGproj{P}{\cA})= \Gproj(\cA).
\]
In particular, this holds if $\cB=\Lambda\text{-}\md$ or $\cB=\Lambda\text{-}\Md$ for an Iwanaga-Gorenstein ring $\Lambda$.
\end{Corollary}

For an abelian category $\cA$ we let $\Omega^{\infty}(\cA)$ denote the collection of objects $A\in \cA$ such that there exists an exact sequence $0\to A\to Q_0\to Q_{-1}\to \cdots$ with $Q_i\in \cA$ projective for all $i\leq 0$. 

\begin{Example}\label{Example:9}
Let $k$ be a field, let $\Lambda_1$ be a finite-dimensional algebra over $k$, and let $\Lambda_2$ be a left coherent $k$-algebra. From Example \ref{Example:5} we have that 
\begin{align*}
 & \Gproj(\relGproj{P}{(\Lambda_1\otimes_k \Lambda_2)\text{-}\md}  =  \{M\in (\Lambda_1\otimes_k \Lambda_2)\text{-}\md \mid \\
 & \text{ } _{\Lambda_1}M\in \Gproj(\Lambda_1\text{-}\Md) \text{ and } _{\Lambda_2}(\Hom_k(\Lambda_1,k)\otimes_{\Lambda_1} M)\in \Gproj(\Lambda_2\text{-}\md)\}.
\end{align*} 
If $\Omega^{\infty}(\Lambda\text{-}\Md)\subset \Gproj(\Lambda\text{-}\Md)$ or 
\begin{multline*}
\Gproj(\Lambda_2\text{-}\md)=\{M\in \Lambda_2\text{-}\md \mid \Ext^1_{\Lambda}(M,M')=0 \\   \text{ for all } M' \text{ satisfying }\pdim M'<\infty \}.
\end{multline*}
then by Theorem \ref{Theorem:4} we have
\[
\Gproj((\Lambda_1\otimes_k \Lambda_2)\text{-}\md)= \Gproj(\relGproj{P}{(\Lambda_1\otimes_k \Lambda_2)\text{-}\md}).
\]
In particular, the equality holds if $\Lambda_1$ or $\Lambda_2$ is Iwanaga-Gorenstein. This description of $\Gproj((\Lambda_1\otimes_k \Lambda_2)\text{-}\md)$ has previously been obtained in \cite{She16}, but it was only shown to hold under the assumption that $\Lambda_1$ is Iwanaga-Gorenstein. 
\end{Example}

\begin{Example}\label{Example:10}
Let $k$ be a field, let $\Lambda_1$ be a finite-dimensional algebra over $k$, and let $\Lambda_2$ be a $k$-algebra. From Example \ref{Example:6} we get that if $\Omega^{\infty}(\Lambda\text{-}\Md)\subset \Gproj(\Lambda\text{-}\Md)$ or
\begin{multline*}
\Gproj(\Lambda_2\text{-}\Md)=\{M\in \Lambda_2\text{-}\Md\mid \Ext^1_{\Lambda}(M,M')=0 \\
  \text{ for all } M' \text{ satisfying }\pdim M'<\infty \}
\end{multline*} then the criteria in Theorem \ref{Theorem:4} holds, and therefore
\begin{align*}
\Gproj((\Lambda_1\otimes_k \Lambda_2)\text{-}\Md) 
= & \{M\in (\Lambda_1\otimes_k \Lambda_2)\text{-}\Md \mid \text{ } _{\Lambda_1}M\in \Gproj(\Lambda_1\text{-}\Md) \\ &\text{ and } _{\Lambda_2}(\Hom_k(\Lambda_1,k)\otimes_{\Lambda_1} M)\in \Gproj(\Lambda_2\text{-}\Md)\}.
\end{align*} 
In particular, this equality holds if $\Lambda_1$ or $\Lambda_2$ are Iwanaga-Gorenstein.
\end{Example}

Since $\Gproj(\relGproj{P}{\cA})$ is closed under direct summands and contains all the projective objects, the projectively stable category $\underline{\Gproj(\relGproj{P}{\cA})}$ is a thick triangulated subcategory of $\underline{\Gproj(\cA )}$.

\begin{Definition}\label{Definition:12}
We define the \emphbf{Gorenstein discrepancy category} of $P$ to be the Verdier quotient $\Discr{P}{\cA}= \underline{\Gproj(\cA )}/\underline{\Gproj(\relGproj{P}{\cA})}$.
\end{Definition}

The triangulated category $\Discr{P}{\cA}$ measures how far $\Gproj(\relGproj{P}{\cA})$ is from $\Gproj(\cA)$. The following example shows that the Gorenstein discrepancy category can be nonzero. 

\begin{Example}\label{Example:11}
Let $k$ be a field, and let $\Lambda_1$ be the path algebra of the quiver
\begin{equation}\label{eq:12}
\begin{tikzpicture}[description/.style={fill=white,inner sep=2pt}]
\matrix(m) [matrix of math nodes,row sep=2.5em,column sep=2.0em,text height=1.5ex, text depth=0.25ex] 
{ 1 & 2 \\ };
\path[->]
(m-1-1) edge node[auto] {$\alpha$} 	    													(m-1-2)
(m-1-2) edge[loop above] node[auto] {$\beta$} 	    	                  (m-1-2);	
\end{tikzpicture}
\end{equation}
with relations $\beta^2 = \beta \circ \alpha = 0$. Let $e_1$ and $e_2$ be the two primitive idempotents of $\Lambda_1$. Note that $\Gproj(\Lambda_1 \text{-}\md)=\Proj(\Lambda_1 \text{-}\md)$. In fact, up to isomorphism the only indecomposable $\Lambda_1$-modules are the two simple modules $S_1$ and $S_2$ concentrated in vertex $1$ and $2$, the two projective modules $P_1=\Lambda e_1$ and $P_2=\Lambda e_2$, and the two injective modules $I_1=\Hom_k(e_1\Lambda,k)$ and $I_2=\Hom_k(e_2\Lambda,k)$. Furthermore, we have an equality $I_2=S_1$. Now since $I_1$ and $I_2$ are injective but not projective, they can't be Gorenstein projective. Also, $S_2$ is not Gorenstein projective since there exists a non-split exact sequence
\[
0\to P_1\to I_2\to S_2\to 0
\]
This shows that $\Gproj(\Lambda_1 \text{-}\md)=\Proj(\Lambda_1 \text{-}\md)$. Now let $\Lambda_2$ be a finite-dimensional $k$-algebra. A module $M\in (\Lambda_1\otimes_k \Lambda_2)\text{-}\md$ can be identified with a representation 
\[
\begin{tikzpicture}[description/.style={fill=white,inner sep=2pt}]
\matrix(m) [matrix of math nodes,row sep=2.5em,column sep=2.0em,text height=1.5ex, text depth=0.25ex] 
{ M_1 & M_2 \\ };
\path[->]
(m-1-1) edge node[auto] {$u$} 	    													(m-1-2)
(m-1-2) edge[loop above] node[auto] {$v$} 	    	                  (m-1-2);	
\end{tikzpicture}
\]
where $M_1,M_2\in \Lambda_2\text{-}\md$ and $u,v$ are morphisms of $\Lambda_2$-modules satisfying $v^2=0$ and $v\circ u=0$. Let
\begin{align*}
& f^*:=\res^{\Lambda_1\otimes_k \Lambda_2}_{\Lambda_2}\colon (\Lambda_1\otimes_k \Lambda_2) \text{-}\md \to \Lambda_2 \text{-}\md \\
& f_!:=\Lambda_1\otimes_k -\colon \Lambda_2 \text{-}\md \to (\Lambda_1\otimes_k \Lambda_2) \text{-}\md \\
& \nu_1:=\Hom_k(\Lambda_1,k)\otimes_{\Lambda_1} -\colon (\Lambda_1\otimes_k \Lambda_2)\text{-} \md \to (\Lambda_1\otimes_k \Lambda_2) \text{-} \md.
\end{align*}
and
\begin{align*}
& g^*:=\res^{\Lambda_1\otimes_k \Lambda_2}_{\Lambda_1}\colon (\Lambda_1\otimes_k \Lambda_2) \text{-}\md \to \Lambda_1 \text{-}\md \\
& g_!:=\Lambda_2\otimes_k -\colon \Lambda_1 \text{-}\md \to (\Lambda_1\otimes_k \Lambda_2) \text{-}\md \\
& \nu_2:=\Hom_k(\Lambda_2,k)\otimes_{\Lambda_2} -\colon (\Lambda_1\otimes_k \Lambda_2)\text{-} \md \to (\Lambda_1\otimes_k \Lambda_2) \text{-} \md.
\end{align*}
be two adjoint pairs with Nakayama functors as in Example \ref{finitely presented}. Let $P_1:=f_!\circ f^*$ and $P_2:=g_!\circ g^*$. We have that 
\begin{multline*}
\Gproj(\relGproj{P_1}{(\Lambda_1\otimes_k \Lambda_2)\text{-}\md}) 
=\{M\in (\Lambda_1\otimes_k \Lambda_2)\text{-}\md \mid \\
\text{ } _{\Lambda_1}M\in \Gproj(\Lambda_1\text{-}\md) \text{ and } _{\Lambda_2}(\Hom_k(\Lambda_1,k)\otimes_{\Lambda_1} M)\in \Gproj(\Lambda_2\text{-}\md)\}
\end{multline*} 
and
\begin{multline*}
\Gproj(\relGproj{P_2}{(\Lambda_1\otimes_k \Lambda_2)\text{-}\md}) 
=\{M\in (\Lambda_1\otimes_k \Lambda_2)\text{-}\md \mid \\
\text{ } _{\Lambda_2}M\in \Gproj(\Lambda_2\text{-}\md) \text{ and } _{\Lambda_1}(\Hom_k(\Lambda_2,k)\otimes_{\Lambda_2} M)\in \Gproj(\Lambda_1\text{-}\md)\}
\end{multline*} 
as in Example \ref{Example:9}. Note that $_{\Lambda_1}M\in \Gproj(\Lambda_1\text{-}\md)=\Proj(\Lambda_1\text{-}\md)$ if and only if the following holds: 
\begin{enumerate}
\item $u$ is a monomorphism;
\item $\im u \cap \im v = (0)$;
\item $\im u \oplus \im v = \Ker v$.
\end{enumerate} 
Also, a simple computation shows that
\[
_{\Lambda_2}(\Hom_k(\Lambda_1,k)\otimes_{\Lambda_1} M)= \Coker u \oplus \Coker v.
\] 
Hence, $M\in \Gproj(\relGproj{P_1}{(\Lambda_1\otimes_k \Lambda_2) \text{-}\md})$ if and only the following holds: 
\begin{enumerate}
\item $u\colon M_1\to M_2$ is a monomorphism;
\item $\im u\cap \im v=(0)$;
\item $\im u\oplus \im v = \Ker v$;
\item $\Coker u, \Coker v\in \Gproj(\Lambda_2\text{-}\md)$. 
\end{enumerate}
Also, $M\in \Gproj(\relGproj{P_2}{(\Lambda_1\otimes_k \Lambda_2)\text{-}\md})$ if and only the following holds:
\begin{enumerate}
\item $M_1,M_2\in \Gproj(\Lambda_2\md)$;
\item $1\otimes_{}u$ is a monomorphism; 
\item $\im (1\otimes_{}u)\cap \im (1\otimes_{}v) = (0)$; 
\item $\im (1\otimes_{}u)\oplus \im (1\otimes_{}v) = \Ker (1\otimes_{}v)$.
\end{enumerate}
where 
\begin{align*}
& 1\otimes_{}u\colon \Hom_k(\Lambda_2,k)\otimes_{\Lambda_2}M_1\to \Hom_k(\Lambda_2,k)\otimes_{\Lambda_2}M_2 \\
& 1\otimes v\colon \Hom_k(\Lambda_2,k)\otimes_{\Lambda_2}M_2\to \Hom_k(\Lambda_2,k)\otimes_{\Lambda_2}M_2.
\end{align*}
Now set $\Lambda_2:=\Lambda_1\op$, and let $Q_2 = \Lambda_2e_2$ and $J_2= \Hom_k(e_2\Lambda_2,k)$ be the projective and injective left $\Lambda_2$-module corresponding to vertex 2. Furthermore, let $s\colon Q_2\to Q_2$ be a nonzero morphism satisfying $s^2=0$ (there exists a unique one up to scalars).  Let $M\in \Lambda_1\otimes \Lambda_2\text{-}\md$ be given by $M_1=0$, $M_2= Q_2$ and $v=s$. Under the isomorphism $\Hom_k(\Lambda_2,k)\otimes_{\Lambda_2}Q_2\cong J_2$ the map $s$ corresponds to a nonzero map $t\colon J_2\to J_2$ satisfying $t^2=0$. There exists a unique such map up to scalars, and it also satisfies $\im t = \ker t$. This shows that $M\in \Gproj(\relGproj{P_2}{(\Lambda_1\otimes_k \Lambda_2)\text{-}\md})$, and $M$ is therefore Gorenstein projective in $(\Lambda_1\otimes_k \Lambda_2)\text{-}\md$. On the other hand, we have that $\im s \neq \Ker s$, and hence $M\notin \Gproj(\relGproj{P_1}{(\Lambda_1\otimes_k \Lambda_2)\text{-}\md})$. This shows that the discrepancy category corresponding to $P_1$ is nonzero.  
\end{Example}

We end this section with a result on the Gorenstein projective dimension of $\cA$.

\begin{Proposition}\label{Proposition:5}
We have the inequality 
\[
\glGpdim \cA \leq \glGpdim \cB + \dim_{\relGproj{P}{\cA}}\cA.
\]
\end{Proposition}

\begin{proof}
It is obviously true if $\glGpdim \cB=\infty$ or $\dim_{\relGproj{P}{\cA}}\cA=\infty$. We therefore assume $\glGpdim \cB =n<\infty$ and $\dim_{\relGproj{P}{\cA}}\cA = m<\infty$. Let $A\in \cA$ be arbitrary, and let
\[
0\to K\xrightarrow{i} Q_{n+m}\xrightarrow{s_{n+m}}Q_{n+m-1}\xrightarrow{s_{n+m-1}}\cdots \xrightarrow{s_2}Q_1\xrightarrow{s_1}A\to 0
\]
be an exact sequence in $\cA$ with $Q_j$ projective for $1\leq j\leq n+m$. Since $Q_j$ is in $\relGproj{P}{\cA}$ and $\dim_{\relGproj{P}{\cA}}A\leq m$, we get that $\Ker s_j\in \relGproj{P}{\cA}$ for $j\geq m$. In particular, this implies that the sequence
\begin{multline*}
0\to f^*\nu(K)\xrightarrow{f^*\nu (i)}f^*\nu (Q_{n+m})\xrightarrow{f^*\nu(s_{n+m})}\cdots \\
 \xrightarrow{f^*\nu(s_{m+2})}f^*\nu (Q_{m+1})\xrightarrow{}f^*\nu (\Ker s_m)\to 0
\end{multline*}
is exact. Since $f^*\nu (Q_j)$ is projective in $\cB$ and $\Gpdim f^*\nu (\Ker s_m)\leq n$, we get that $f^*\nu (K)\in \Gproj(\cB )$. Hence, $K\in \Gproj(\relGproj{P}{\cA}) = \Gproj(\cA )$, and the claim follows.
\end{proof}

\section{Application to functor categories}\label{Application to functor categories}
Our goal in this section is to compute $\Gproj(\cB^{\cC})$ in examples using the theory we have developed. 

\subsection{Preliminaries}\label{Subsection Preliminaries}

Let $k$ be a commutative ring, and let $\cC$ be a small $k$-linear category. Recall that a right $\cC$-module is a $k$-linear functor $\cC\op\to k\text{-}\Md$. We let $\Md\text{-}\cC$ denote the categories of right $\cC$-modules.  A right $\cC$-module $M$ is called \emphbf{finitely presented} if there exists an exact sequence  
\[
\oplus_{i=1}^m\cC(-,c_i) \to \oplus_{j=1}^n\cC(-,d_j)\to M\to 0 
\]
in $\Md \text{-}\cC$ for objects $c_i, d_j\in \cC$. The category of finitely presented right $\cC$-modules is denoted by $\md \text{-}\cC$. Dually, the category of left $\cC$-modules and finitely presented left $\cC$-modules are $\Md \text{-}\cC\op$ and $\md \text{-}\cC\op$, respectively. 

Let $\cB$ be a $k$-linear abelian category, and let $\cB^{\cC}$ denote the category of $k$-linear functors from $\cC$ to $\cB$. Up to isomorphism there exists a unique functor
\begin{align*}
- \otimes_{\cC}- \colon (\md\text{-} \cC) \otimes \cB^{\cC} \to \cB 
\end{align*} 
such that $\cC(c,-)\otimes_{\cC}F=F(c)$ and the induced functor $-\otimes_{\cC} F\colon\md\text{-} \cC \to \cB$ is right exact for all $F\in \cB^{\cC}$, see chapter 3 in \cite{Kel05} or \cite{OR70} for details. If $\cC=k$ we get a functor 
\[
-\otimes_{k}- \colon (k\text{-}\md) \otimes \cB \to \cB
\]
For  $N\in \cC\text{-}\md$ and $B\in \cB$ we have a functor $N\otimes_k B\in \cB^{\cC}$ given by $c\mapsto N(c)\otimes_k B$. If furthermore $M\in \md\text{-}\cC$ then we get a natural isomorphism 
\[
M\otimes_{\cC}(N\otimes_k B) \cong (M\otimes_{\cC}N)\otimes_k B
\]
see (3.23) in \cite{Kel05}.

We use the same terminology as in \cite{DSS17} in the following definition.
\begin{Definition}\label{Definition:12,5}
Let $\cC$ be a small $k$-linear category.
\begin{enumerate}
\item  $\cC$ is \emphbf{locally bounded} if for any object $c\in \cC$ there are only finitely many objects in $\cC$ mapping nontrivially in and out of $c$. This means that for each $c\in \cC$ we have
\begin{align*}
& \cC(c,c')\neq 0 \quad \text{for only finitely many } c'\in \cC \\
& \cC(c'',c)\neq 0 \quad \text{for only finitely many } c''\in \cC.
\end{align*}
\item  $\cC$ is \emphbf{Hom-finite} if $\cC(c,c')$ is a finitely generated projective $k$-module for all $c,c'\in \cC$.
\end{enumerate}
\end{Definition}
If $\cC$ is locally bounded and Hom-finite, and $M\in \Md\text{-}\cC$ satisfies
\begin{enumerate}
\item $M(c)$ is a finitely generated projective $k$-module for all $c\in \cC$
\item $M(c)\neq 0$ for only finitely many $c\in \cC$
\end{enumerate}
then it follows from \cite[Lemma 5.2.2]{Kva17} that $M\in \md\text{-}\cC$.

Let $k(\ob\text{-} \cC)$ be the category with the same objects as $\cC$, and with morphisms
\[
k(\ob \text{-}\cC)(c_1,c_2) =
\begin{cases}
0 & \text{if $c_1\neq c_2$},\\
k & \text{if $c_1=c_2$}.
\end{cases}
\]
The functor category $\cB^{k(\ob\text{-} \cC)}$ is just a product of copies of $\cB$, indexed over the objects of $\cC$. Let $i\colon k(\ob\text{-} \cC)\to \cC$ be the inclusion. We have functors 
\begin{align*}
& i_!\colon \cB^{k(\ob\text{-} \cC)}\to \cB^{\cC} \quad i_!((B^c)_{c\in \cC})= \bigoplus_{c\in \cC}\cC(c,-)\otimes_k B^c \\
& i^*\colon \cB^{\cC}\to \cB^{k(\ob\text{-} \cC)} \quad \quad i^*(F)= (F(c))_{c\in \cC} \\
& \nu\colon \cB^{\cC}\to \cB^{\cC} \quad \nu(F)= D(\cC)\otimes_{\cC}F 
\end{align*} 
where $D=\Hom_k(-,k)$ and $(D(\cC)\otimes_{\cC}F)(c)= D(\cC(c,-))\otimes_{\cC}F$, see Subsection 5.3 in \cite{Kva17} for details.

\begin{Theorem}[Theorem 5.3.3 in \cite{Kva17}]\label{Nakayama functor on functor categories}
Let $\cC$ be a small, $k$-linear, locally bounded and Hom-finite category, let $\cB$ be a $k$-linear abelian category, and let $i_!$, $i^*$ and $\nu$ be as above. Then $\nu$ is a Nakayama functor relative to $i_!\dashv i^*$.
\end{Theorem}

\begin{Theorem}[Theorem 5.3.4 in \cite{Kva17}]\label{Gorenstein locally bounded Hom-finite}
Let $\cC$ be a small, $k$-linear, locally bounded, and Hom-finite category, and let $i_!$, $i^*$ and $\nu$ be as above with $\cB=k\text{-}\Md$. Then  
\[
  \sup_{c\in \cC}(\pdim D(\cC(-,c)))<\infty \quad \text{and} \quad \sup_{c\in \cC}(\pdim D(\cC(c,-)))<\infty
\] 
if and only if the endofunctor $P=i_!\circ i^*\colon \cC\text{-}\Md\to \cC\text{-}\Md$ is Iwanaga-Gorenstein. In this case we have that
\[
\sup_{c\in \cC}(\pdim D(\cC(-,c))) = \sup_{c\in \cC}(\pdim D(\cC(c,-))).
\]
and this number is equal to the Gorenstein dimension of the functor $P$.
\end{Theorem}

\subsection{Properties of locally bounded and Hom-finite categories}

In this subsection we fix a small, $k$-linear, locally bounded and Hom-finite category $\cC$  and a $k$-linear abelian category $\cB$. Let $M$ be a finitely presented right $\cC$-module. Since
\[
((M\otimes_{\cC}-)\circ i_!\circ i^*)(F)\cong  \bigoplus_{c\in \cC}M(c)\otimes_k F(c)
\]
it follows that the functor $(M\otimes_{\cC}-)\circ i_!\circ i^*\colon \cB^{\cC}\to \cB^{\cC}$ is exact if $M(c)$ is a finitely generated projective $k$-module for all $c\in \cC$. By Proposition \ref{Right and left derived functors} part \ref{Right and left derived functors:2} we have that $i_!$ is adapted to $M\otimes_{\cC}-$, and hence the left derived functor
\[
\Tor^{\cC}_n(M,-):=L_n(M\otimes_{\cC}-).
\]
exists. 

\begin{Lemma}\label{Lemma:23}
Let $0\to M_3 \xrightarrow{f} M_2\xrightarrow{g} M_1\to 0$ be an exact sequence of finitely presented right $\cC$-modules, and assume $M_i(c)$ is a finitely generated projective $k$-module for all $i$ and all $c\in \cC$. Then there exists a long exact sequence of functors
	\begin{multline*}
\cdots \to \Tor_{i+1}^{\cC}(M_1,-)\to \Tor_i^{\cC}(M_3,-)\to \Tor_i^{\cC}(M_2,-)\to \Tor_{i}^{\cC}(M_1,-)\to \\
\cdots \to \Tor_1^{\cC}(M_1,-)\to (M_3\otimes_{\cC}-) \xrightarrow{f\otimes 1} (M_2\otimes_{\cC}-)  \xrightarrow{g\otimes 1} (M_1\otimes_{\cC}-)\to 0.
\end{multline*}  
\end{Lemma}

\begin{proof}
Consider the sequence 
\[
(M_3\otimes_{\cC}-)\xrightarrow{f\otimes 1} (M_2\otimes_{\cC}-)  \xrightarrow{g\otimes 1} (M_1\otimes_{\cC}-)
\]
of functors. Evaluating at the object $i_!i^*(F)= \bigoplus_{c\in \cC}\cC(c,-)\otimes_k F(c)$ gives the exact sequence
\[
0\to \bigoplus_{c\in \cC}M_3(c)\otimes_k F(c) \xrightarrow{f\otimes 1} \bigoplus_{c\in \cC}M_2(c)\otimes_k F(c)\xrightarrow{g\otimes 1} \bigoplus_{c\in \cC}M_1(c)\otimes_k F(c)\to 0
\]
The claim follows therefore by Lemma \ref{Lemma:5**}. 
\end{proof}

From now on we let $P_{\cB^{\cC}}=i_!\circ i^*$ denote the endofunctor on $\cB^{\cC}$ and $P_{\cC\text{-}\Md}$ the endofunctor on $\cC\text{-}\Md$ in Theorem \ref{Gorenstein locally bounded Hom-finite}. 

\begin{Lemma}\label{Lemma:24}
Assume $P_{\cC\text{-}\Md}$ is $n$-Gorenstein. Then $P_{\cB^{\cC}}$ is $m$-Gorenstein where $m\leq n$.
\end{Lemma}

\begin{proof}
Let $c\in \cC$ be arbitrary. By Theorem \ref{Gorenstein locally bounded Hom-finite} there exists an exact sequence 
\[
0\to M_n\to M_{n-1}\to \cdots \to M_1\to M_0\to D(\cC(c,-))\to 0
\]
in $\md\text{-} \cC$ where $M_i$ are projective. By Lemma \ref{Lemma:23} and dimension shifting we get that 
\[
\Tor_j^{\cC}(D(\cC(c,-)),-)\colon \cB^{\cC}\to \cB
\]
is $0$ for all $j\geq n+1$. Since $c\in \cC$ was arbitrary we get that 
\[
L_j\nu= \Tor_j^{\cC}(D(\cC),-)\colon \cB^{\cC}\to \cB^{\cC}
\]
is $0$ for $j\geq n+1$. Dually, we also have that $R^j\nu^-=0$ for $j\geq n+1$. The claim follows.
\end{proof}

A small $k$-linear category $\cC'$ is called \emphbf{left Gorenstein} if 
\[
\glGpdim\cC'\text{-}\Md< \infty. 
\]
Note that by \cite[Theorem 4.16]{Bel00} the category $\cC'$ is left Gorenstein if and only if  $\glGidim \cC'\text{-}\Md <\infty$, where $\glGidim \cC'\text{-}\Md$ is the global Gorenstein injective dimension of $\cC'\text{-}\Md$. Furthermore, if $\cC'$ is left Gorenstein then 
\[
\glGpdim\cC'\text{-}\Md= \glGidim \cC'\text{-}\Md
\] 
and $\cC'$ is called \emphbf{left} $m$\emphbf{-Gorenstein} if this common number is $m$.  

\begin{Theorem}\label{Theorem:6}
Let $\cC'$ be a small $k$-linear category, and assume $\cC'$ is left $m$-Gorenstein. Furthermore, assume the endofunctor $P_{\cC\text{-}\Md}$ is $n$-Gorenstein. Then the category $\cC'\otimes \cC$ is left $p$-Gorenstein where $p\leq m + n$. 
\end{Theorem}

\begin{proof}
This follows from Proposition \ref{Proposition:5}, Theorem \ref{Nakayama functor on functor categories}, and Lemma \ref{Lemma:24} applied to $(\cC'\otimes_k \cC)\text{-}\Md= (\Md\text{-}\cC')^{\cC}$.
\end{proof}

It would be interesting to know when the equality $p=m+n$ in Theorem \ref{Theorem:6} holds.

\begin{Remark}\label{Remark:3}
Following the conventions in \cite{DSS17}, we say that the category $\cC$ has a Serre functor relative to $k$ if there exists an equivalence $S\colon \cC\to \cC$ together with a natural isomorphism
\[
\cC(c_1,c_2)\cong D(\cC(c_2, S(c_1)))
\] 
for all $c_1,c_2\in \cC$. This implies in particular that $P_{\cC\text{-}\Md}$ is $0$-Gorenstein. Theorem \ref{Theorem:6} therefore gives a partial generalization of \cite[Theorem 4.6]{DSS17}. 
\end{Remark}

\subsection{Monic representations of a quiver}\label{Monic representations of a quiver}

Let $Q=(Q_0,Q_1,s,t)$ be a quiver (not necessarily finite) such that for each vertex $i\in Q_0$ there are only finitely many paths starting in $i$ and only finitely many paths ending in $i$. Let $\cC=kQ$ be the $k$-linearization of $Q$. Obviously, $kQ$ is a Hom-finite and locally bounded category. An object $F\in \cB^{kQ}$ is a representation of $Q$ over $\cB$, given by the datum $F=(F(i),f_{\alpha}, i\in Q_0, \alpha\in Q_1)$, where $F(i)\in \cB$ and $f_{\alpha}\colon F(s(\alpha ))\to F(t(\alpha ))$ are morphisms in $\cB$. A morphism 
\[
\phi\colon (F(i),f_{\alpha}, i\in Q_0, \alpha\in Q_1)\to (F'(i),g_{\alpha}, i\in Q_0, \alpha\in Q_1)
\]
 is given by morphisms $\phi_i\colon F(i)\to F'(i)$ for each $i\in Q_0$, such that the diagram
\begin{equation*}
\begin{tikzpicture}[description/.style={fill=white,inner sep=2pt}]
\matrix(m) [matrix of math nodes,row sep=2.5em,column sep=5.0em,text height=1.5ex, text depth=0.25ex] 
{ F(s(\alpha )) & F(t(\alpha )) \\
  F'(s(\alpha )) & F'(t(\alpha )) \\};
\path[->]
(m-1-1) edge node[auto] {$f_{\alpha}$} 	    													    (m-1-2)
(m-2-1) edge node[auto] {$g_{\alpha}$} 	    													    (m-2-2)

(m-1-1) edge node[auto] {$\phi_{s(\alpha )}$} 	    								    (m-2-1)
(m-1-2) edge node[auto] {$\phi_{t(\alpha )}$} 	    									  (m-2-2);
\end{tikzpicture}
\end{equation*} 
commutes for each $\alpha\in Q_1$. We let $kQe_i$ and $e_ikQ$ denote the representable functors $kQ(i,-)$ and $kQ(-,i)$.

\begin{Definition}\label{Definition:13}
A representation $F=(F(i),f_{\alpha}, i\in Q_0, \alpha\in Q_1)$ is \emphbf{monic} if 
\[
(f_{\alpha})_{t(\alpha)=i}\colon \bigoplus_{t(\alpha)=i}F(s(\alpha))\to F(i)
\]
is a monomorphism for all $i\in Q_0$.
\end{Definition} 

Let $\Mon (Q,\cB )$ denote the subcategory of $\cB^{kQ}$ consisting of the monic representations. It was considered in \cite{LZ13} for $Q$ a finite acyclic quiver, $k$ a field, and $\cB =\Lambda\text{-}\md$ the category of finite dimensional modules over a finite dimensional algebra $\Lambda$. It was also considered in \cite{EHS13} for $Q$ a left rooted quiver and $\cB = \Lambda\text{-}\Md$ for $\Lambda$ an arbitrary ring. In both cases it is used to give a description of the Gorenstein projective objects in $\cB^{kQ}$. We recover this description using the theory we have developed.

\begin{Proposition}\label{Proposition:9}
The following holds:
\begin{enumerate}
	\item\label{Proposition:9,1} The endofunctor $P_{kQ\text{-}\Md}$ is $m$-Gorenstein where $m\leq 1$;
	\item\label{Proposition:9,2} A representation $F\in \cB^{kQ}$ is monic if and only if it is Gorenstein $P_{\cB^{kQ}}$-projective.
\end{enumerate}  
\end{Proposition}

\begin{proof}
Fix a vertex $i\in Q_0$, and let $S_i\in \Md\text{-}kQ$ be the representation 
$$
S_i(j)=
\begin{cases}
k & \text{if } i=j \\
0 & \text{if } i\neq j.
\end{cases} 
$$
We have a projective resolution of $S_i$ given by
\begin{equation}\label{Equation:5}
0\to \bigoplus_{t(\alpha) =i}e_{s(\alpha )}kQ\to e_i kQ\to S_i\to 0
\end{equation}
where the morphism $e_{s(\alpha )}kQ\to e_i kQ$ is induced from $\alpha\colon s(\alpha)\to i$. This shows that $\pdim S_i\leq 1$ for all $i\in Q_0$. Also, $D(kQe_i)$ has a filtration 
\begin{equation}\label{Equation:6}
0=M_0\subset M_1\subset \cdots \subset M_n=D(kQe_i)
\end{equation}
in $\md\text{-} kQ$ such that $M_{i+1}/M_i\cong S_{j_i}$ for vertices $j_0,j_1,\cdots j_{n-1}\in Q_0$. Therefore, we get that $\pdim D(kQe_i) \leq 1$ for all $i\in Q_0$. Dually, the same argument applied to $Q\op$ shows that $\pdim D(e_ikQ)\leq 1$ for all $i\in Q_0$. This proves that the endofunctor $P_{kQ\text{-}\Md}$ is $m$-Gorenstein where $m\leq 1$.

We now describe the objects which are Gorenstein $P_{\cB^{kQ}}$-projective. By Lemma \ref{Lemma:24} we know that $P_{\cB^{kQ}}$ is Iwanaga-Gorenstein of dimension $0$ or $1$. Hence, by Theorem \ref{Theorem:2} and Theorem \ref{Theorem:2.5} part \ref{Theorem:2.5:2} the Gorenstein $P_{\cB^{kQ}}$-projective functors are precisely the functors $F\in \cB^{kQ}$ such that 
\[
\Tor^{kQ}_1(D(kQe_i),F)=0
\]
 for all $i\in Q_0$. Now for all $i\in Q_0$ we have an exact sequence
\begin{equation}\label{Equation:7}
0\to S_i \to D(kQe_i) \to \bigoplus_{t(\alpha) =i}D(kQe_{s(\alpha )})\to 0
\end{equation}
obtained by applying $D(-)$ to the sequence \eqref{Equation:5} with $Q$ replaced by $Q\op$. Hence, we get that 
\begin{multline*}
\Tor^{kQ}_1(D(kQe_i),F)=0 \text{ }\forall i\in Q_0 
 \implies \Tor^{kQ}_1(S_i,F)=0 \text{ }\forall i\in Q_0
\end{multline*}
by tensoring $F$ with the sequence in \eqref{Equation:7} and using Lemma \ref{Lemma:23}. Conversely, from the filtration \eqref{Equation:6} we get that 
\begin{multline*}
\Tor^{kQ}_1(S_i,F)=0 \text{ }\forall i\in Q_0
 \implies \Tor^{kQ}_1(D(kQe_i),F)=0 \text{ }\forall i\in Q_0
\end{multline*}
by repeated use of Lemma \ref{Lemma:23}. Hence, $F$ is Gorenstein $P_{\cB^{kQ}}$-projective if and only if $\Tor^{kQ}_1(S_i,F)=0$ for all $i\in Q_0$. Tensoring the sequence \eqref{Equation:5} with $F$ gives the exact sequence
\begin{equation}\label{Equation:8}
0\to \Tor^{kQ}_1(S_i,F)\to \bigoplus_{t(\alpha) =i}F(s(\alpha))\xrightarrow{(f_{\alpha})_{t(\alpha)=i}} F(i)\to S_i\otimes_{kQ}F\to 0.
\end{equation}
Hence, $F$ is Gorenstein $P_{\cB^{kQ}}$-projective if and only if it is monic.
\end{proof}

\begin{Proposition}\label{Proposition:10}
Assume $\cB$ has enough projectives. The following holds:
\begin{enumerate}
	\item\label{Proposition:10,1} A functor $F=(F(i),f_{\alpha}, i\in Q_0, \alpha\in Q_1)\in \cB^{kQ}$ is Gorenstein projective if and only if it is monic and the cokernel of the map
\[
(f_{\alpha})_{t(\alpha)=i}\colon  \bigoplus_{t(\alpha)=i}F(s(\alpha))\to F(i)
\]
is Gorenstein projective in $\cB$ for all $i\in Q_0$; 
\item\label{Proposition:10,2} If $F$ is Gorenstein projective in $\cB^{kQ}$, then $F(i)$ is Gorenstein projective in $\cB$ for all $i\in Q_0$. 
\end{enumerate} 
\end{Proposition}

\begin{proof}
We know by Corollary \ref{Gorenstein adjoint pairs lifts Gorenstein projectives} and Proposition \ref{Proposition:9} that $F$ is Gorenstein projective if and only if it is monic and $D(kQe_i)\otimes_{kQ}F\in \Gproj(\cB)$ for all $i\in Q_0$. Assume $F$ is monic, and consider the exact sequence \eqref{Equation:7}. Tensoring with $F$ gives an exact sequence
\[
0\to S_i\otimes_{kQ}F \to D(kQe_i)\otimes_{kQ}F \to (\bigoplus_{t(\alpha) =i}D(kQe_{s(\alpha )}))\otimes_{kQ}F\to 0
\]
since $\Tor^1_{kQ}(\bigoplus_{t(\alpha) =i}D(kQe_{s(\alpha )}),F) =0$. Hence, we get that 
\[
D(kQe_i)\otimes_{kQ}F\in \Gproj(\cB) \text{ }\forall i\in Q_0 \implies S_i\otimes_{kQ} F\in \Gproj(\cB) \text{ }\forall i\in Q_0
\]
since $\Gproj(\cB )$ is closed under kernels of epimorphisms. Also, from the filtration in  \eqref{Equation:6} we have an exact sequence
\[
0\to M_i\to M_{i+1}\to S_{j_i}\to 0
\]
for each $0\leq i\leq n-1$. Tensoring this with $F$ gives an exact sequence
\[
0\to M_i\otimes_{kQ}F\to M_{i+1}\otimes_{kQ}F\to S_{j_i}\otimes_{kQ}F\to 0
\]
since $\Tor^1_{kQ}(S_{j_i},F)=0$. Therefore,
\[
S_i\otimes_{kQ}F\in \Gproj(\cB) \text{ }\forall i\in Q_0 \implies D(kQe_i)\otimes_{kQ} F\in \Gproj(\cB) \text{ }\forall i\in Q_0
\]
since $\Gproj(\cB )$ is closed under extensions. Hence, a functor $F\in \cB^{kQ}$ is Gorenstein projective if and only if it is monic and $S_i\otimes_{kQ}F\in \Gproj(\cB)$ for all $i\in Q_0$. By the exact sequence in \eqref{Equation:8} we see that $S_i\otimes_{kQ}F$ is the cokernel of the map 
\[
(f_{\alpha})_{t(\alpha)=i}\colon  \bigoplus_{t(\alpha)=i}F(s(\alpha))\to F(i)
\]
 and the claim follows.

For statement \ref{Proposition:10,2}, note that $e_ikQ$ has a filtration $0=M_0\subset M_1\subset \cdots \subset M_{n'}=e_ikQ$ such that $M_{i+1}/M_i\cong S_{j'_i}$ for $j'_0,j'_1,\cdots j'_{n'-1}\in Q_0$. Hence, if $F$ is Gorenstein projective, then $e_ikQ\otimes_{kQ}F\cong F(i)$ is Gorenstein projective for all $i\in Q_0$. This proves the claim.   
\end{proof}

\subsection{More examples}\label{More examples}

In this subsection we calculate the Gorenstein projective objects in examples for representation of quiver with relations over $\cB$.

\begin{Example}\label{Example:12}
Let $\cC$ be the $k$-linear category generated by the quiver 
\[
\cdots \xrightarrow{d_{i+2}} c_{i+1}\xrightarrow{d_{i+1}} c_{i}\xrightarrow{d_{i}} \cdots
\]
with vertex set $\{c_i \mid i\in \bZ/n \bZ\}$ and relations $d_i\circ d_{i+1}=0$. The category $\cB^{\cC}$ can be identified with $n$-periodic complexes over $\cB$ (for $n=0$ this is just unbounded complexes over $\cB$). It was shown in \cite[Proposition 4.12]{DSS17} that $\cC$ has a relative Serre functor $S$ given by $S(c_i)=c_{i-1}$ and $S(d_i)=d_{i-1}$. Therefore, the endofunctor $P_{\cC\text{-}\Md}$ is $0$-Gorenstein. Hence, by Theorem \ref{Theorem:2.5} we get that $\relGproj{P_{\cB^{\cC}}}{\cB^{\cC}} = \cB^{\cC}$. If $\cB$ has enough projectives, then the Gorenstein projective objects in $\cB^{\cC}$ are precisely the functors $F$ such that 
\[
D\cC(c_{i+1},-)\otimes_{\cC}F \cong \cC(-,c_{i})\otimes_{\cC}F \cong F(c_{i})\in \Gproj(\cB )
\]
for all $c_i\in \cC$. Note that for $n=0$ this recovers the description obtained in \cite[Theorem 2.2]{YL11}. Also, if we put $\cX = \relGproj{P_{\cB^{\cC}}}{\cB^{\cC}}$ and $\cY = \Proj(\cB)$ in Theorem \ref{Theorem:3} we recover the result that the collection of $n$-periodic complexes over $\cB$ with projective components form a Frobenius exact category.  
\end{Example}

\begin{Example}\label{Example:13}
Let $\cC$ be the $k$-linear category generated by the quiver 
\[
c_n\xrightarrow{d_n}c_{n-1}\xrightarrow{d_{n-1}}\cdots \xrightarrow{d_{1}}c_0
\]
with relations $d_{i}\circ d_{i+1}=0$ for $1\leq i\leq n-1$. Then $D(\cC(c_i,-))\cong \cC(-,c_{i-1})$ in $\Md\text{-}\cC$ for $1\leq i\leq n$  and $D(\cC(-,c_i))\cong \cC(c_{i+1},-)$ in $\cC\text{-} \Md$ for $0\leq i\leq n-1$. Furthermore, we have an exact sequence
\begin{equation}\label{Equation:9}
0\to \cC(-,c_n)\to \cC(-,c_{n-1})\to \cdots \to \cC(-,c_{0})\to D(\cC(c_{0},-))\to 0
\end{equation}
in $\Md\text{-} \cC$ and an exact sequence
\[
0\to \cC(c_{0},-)\to \cC(c_{1},-)\to \cdots \to \cC(c_n,-)\to D(\cC(-,c_n))\to 0
\]
in $\cC\text{-} \Md$. Hence, the endofunctor $P_{\cC\text{-}\Md}$ is $n$-Gorenstein. Let $F\in \cB^{\cC}$ be a functor. We can identify $F$ with a complex
\[
F(c_n)\xrightarrow{f_n} F(c_{n-1})\xrightarrow{f_{n-1}} \cdots \xrightarrow{f_{1}} F(c_0).
\]
with $n+1$ terms. Tensoring the sequence \eqref{Equation:9} with $F$ gives a sequence 
\[
F(c_n)\xrightarrow{f_n} F(c_{n-1})\xrightarrow{f_{n-1}} \cdots \xrightarrow{f_{1}} F(c_0) \to D\cC(c_{0},-)\otimes_{\cC}F.
\]
By Theorem \ref{Theorem:2.5} part \ref{Theorem:2.5:2} we get that $F$ is Gorenstein $P_{\cB^{\cC}}$-projective if and only if $\Tor^{kQ}_j(D\cC(c_{0},-),F)=0$ for all $1\leq j\leq n$. Since 
\begin{align*} 
& \Tor^{kQ}_j(D\cC(c_{0},-),F)= \Ker f_{j}/\im f_{j+1} \quad \text{for} \quad 1\leq j\leq n-1 \\
& \Tor^{kQ}_n(D\cC(c_{0},-),F)=\Ker f_n
\end{align*}
 it follows that  $F$ is Gorenstein $P_{\cB^{\cC}}$-projective if and only if the sequence 
\begin{equation}\label{Equation:10}
0\to F(c_n)\xrightarrow{f_n} F(c_{n-1})\xrightarrow{f_{n-1}} \cdots \xrightarrow{f_{1}} F(c_0)
\end{equation}
is exact. Now assume $\cB$ has enough projectives. Then $\Gproj(\relGproj{P_{\cB^{\cC}}}{\cB^{\cC}}) = \Gproj(\cB^{\cC})$ by Corollary \ref{Gorenstein adjoint pairs lifts Gorenstein projectives}. Therefore, the Gorenstein projective objects in $\cB^{\cC}$ are precisely the functors $F$ such that sequence \eqref{Equation:10} is exact and 
\begin{align*}
& D(\cC(c_i,-))\otimes_{\cC}F\cong F(c_{i-1})\in \Gproj (\cB) \quad \text{for } 1\leq i \leq n \\
& D(\cC(c_0,-))\otimes_{\cC}F \cong \Coker f_{1}\in \Gproj(\cB).
\end{align*}
\end{Example}

\begin{Example}\label{Example:14}
Let $\cC$ be the $k$-linear category generated by the quiver
\begin{equation*}
\begin{tikzpicture}[description/.style={fill=white,inner sep=2pt}]
\matrix(m) [matrix of math nodes,row sep=2.5em,column sep=5.0em,text height=1.5ex, text depth=0.25ex] 
{ c_1 & c_2 \\
  c_3 & c_4 \\};
\path[->]
(m-1-1) edge node[auto] {$\alpha$} 	    													    (m-1-2)
(m-2-1) edge node[auto] {$\gamma$} 	    													    (m-2-2)

(m-1-1) edge node[auto] {$\mu$} 	    								    (m-2-1)
(m-1-2) edge node[auto] {$\beta$} 	    									  (m-2-2);
\end{tikzpicture}
\end{equation*} 
with relations $\beta \circ \alpha = \gamma \circ \mu$. A functor $F\in \cB^{\cC}$ is just a commutative diagram in $\cB$. Note that $\cC(-,c_4)\cong D\cC(c_1,-)$. Also, there are exact sequences
\begin{align*}
& 0\to \cC(-,c_3)\xrightarrow{\gamma\circ -} \cC(-,c_4)\to D\cC(c_2,-)\to 0 \\
& 0\to \cC(-,c_2)\xrightarrow{\beta\circ -} \cC(-,c_4)\to D\cC(c_3,-)\to 0
\end{align*}
and
\begin{multline*}
0\to \cC(-,c_1)\xrightarrow{\begin{bmatrix}-(\alpha\circ -) \\  \mu\circ -\end{bmatrix}} \cC(-,c_2)\oplus \cC(-,c_3)\xrightarrow{\begin{bmatrix}\beta\circ -&  \gamma\circ -\end{bmatrix}} \cC(-,c_4) \\
\to D\cC(c_4,-)\to 0
\end{multline*}
in $\Md\text{-} \cC$. Since $\cC$ is isomorphic to $\cC\op$ the same holds for $\cC\op$. Hence, the endofunctor $P_{\cC\text{-}\Md}$ is $2$-Gorenstein. By Theorem \ref{Theorem:2.5} part \ref{Theorem:2.5:2} we get that $F\in \cB^{\cC}$ is Gorenstein $P_{\cB^{\cC}}$-projective if and only if $\Tor^{\cC}_j(D(\cC(c_i,-)),F)=0$ for $1\leq j\leq 2$ and $1\leq i\leq 4$. Tensoring $F$ with the exact sequences above shows that $F\in \cB^{\cC}$ is Gorenstein $P_{\cB^{\cC}}$-projective if and only if 
\[
F(c_3)\xrightarrow{F(\gamma)} F(c_4) \quad  \text{and} \quad  F(c_2)\xrightarrow{F(\beta)} F(c_4)
\]
are monomorphisms and the diagram
\begin{equation*}
\begin{tikzpicture}[description/.style={fill=white,inner sep=2pt}]
\matrix(m) [matrix of math nodes,row sep=2.5em,column sep=5.0em,text height=1.5ex, text depth=0.25ex] 
{ F(c_1) & F(c_2) \\
  F(c_3) & F(c_4) \\};
\path[->]
(m-1-1) edge node[auto] {$F(\alpha)$} 	    													    (m-1-2)
(m-2-1) edge node[auto] {$F(\gamma)$} 	    													    (m-2-2)

(m-1-1) edge node[auto] {$F(\mu)$} 	    								    							(m-2-1)
(m-1-2) edge node[auto] {$F(\beta)$} 	    									  						(m-2-2);
\end{tikzpicture}
\end{equation*}
is a pullback square. If $\cB$ has enough projectives, then a functor $F\in \cB^{\cC}$ is Gorenstein projective if and only if it is Gorenstein $P_{\cB^{\cC}}$-projective and 
\begin{align*}
& D\cC(c_1,-)\otimes_{\cC}F \cong F(c_4)\in \Gproj(\cB) \\
& D\cC(c_2,-)\otimes_{\cC}F\cong \Coker (F(c_3)\xrightarrow{F(\gamma)}F(c_4))\in \Gproj(\cB) \\
& D\cC(c_3,-)\otimes_{\cC}F \cong \Coker (F(c_2)\xrightarrow{F(\beta)}F(c_4))\in \Gproj(\cB) \\
& D\cC(c_4,-)\otimes_{\cC}F \cong \Coker (F(c_2)\oplus F(c_3)\xrightarrow{\begin{bmatrix}F(\beta )&  F(\gamma )\end{bmatrix}}F(c_4))\in \Gproj(\cB). 
\end{align*}
\end{Example}

\bibliography{Mybibtex}

\begin{thebibliography}{10}

\bibitem{AB69}
Maurice Auslander and Mark Bridger.
\newblock {\em Stable module theory}.
\newblock Memoirs of the American Mathematical Society, No. 94. American
  Mathematical Society, Providence, R.I., 1969.

\bibitem{AB89}
Maurice Auslander and Ragnar-Olaf Buchweitz.
\newblock The homological theory of maximal {C}ohen-{M}acaulay approximations.
\newblock {\em M\'em. Soc. Math. France (N.S.)}, (38):5--37, 1989.
\newblock Colloque en l'honneur de Pierre Samuel (Orsay, 1987).

\bibitem{AM02}
Luchezar~L. Avramov and Alex Martsinkovsky.
\newblock Absolute, relative, and {T}ate cohomology of modules of finite
  {G}orenstein dimension.
\newblock {\em Proc. London Math. Soc. (3)}, 85(2):393--440, 2002.

\bibitem{BB69}
Michael Barr and Jon Beck.
\newblock Homology and standard constructions.
\newblock In {\em Sem. on {T}riples and {C}ategorical {H}omology {T}heory
  ({ETH}, {Z}\"urich, 1966/67)}, pages 245--335. Springer, Berlin, 1969.

\bibitem{Bel00}
Apostolos Beligiannis.
\newblock The homological theory of contravariantly finite subcategories:
  {A}uslander-{B}uchweitz contexts, {G}orenstein categories and
  (co-)stabilization.
\newblock {\em Comm. Algebra}, 28(10):4547--4596, 2000.

\bibitem{Bel05}
Apostolos Beligiannis.
\newblock Cohen-{M}acaulay modules, (co)torsion pairs and virtually
  {G}orenstein algebras.
\newblock {\em J. Algebra}, 288(1):137--211, 2005.

\bibitem{Bel11}
Apostolos Beligiannis.
\newblock On algebras of finite {C}ohen-{M}acaulay type.
\newblock {\em Adv. Math.}, 226(2):1973--2019, 2011.

\bibitem{BK08}
Apostolos Beligiannis and Henning Krause.
\newblock Thick subcategories and virtually {G}orenstein algebras.
\newblock {\em Illinois J. Math.}, 52(2):551--562, 2008.

\bibitem{BR07}
Apostolos Beligiannis and Idun Reiten.
\newblock Homological and homotopical aspects of torsion theories.
\newblock {\em Mem. Amer. Math. Soc.}, 188(883):viii+207, 2007.

\bibitem{BM07}
Driss Bennis and Najib Mahdou.
\newblock Strongly {G}orenstein projective, injective, and flat modules.
\newblock {\em J. Pure Appl. Algebra}, 210(2):437--445, 2007.

\bibitem{Bue10}
Theo B{\"u}hler.
\newblock Exact categories.
\newblock {\em Expo. Math.}, 28(1):1--69, 2010.

\bibitem{DSS17}
Ivo Dell'Ambrogio, Greg Stevenson, and Jan Stovicek.
\newblock Gorenstein homological algebra and universal coefficient theorems.
\newblock {\em Math. Z.}, 287(3-4):1109--1155, 2017.

\bibitem{EEG08}
E.~Enochs, S.~Estrada, and J.~R. Garcia-Rozas.
\newblock Gorenstein categories and {T}ate cohomology on projective schemes.
\newblock {\em Math. Nachr.}, 281(4):525--540, 2008.

\bibitem{EEG09}
E.~Enochs, S.~Estrada, and J.~R. Garcia-Rozas.
\newblock Injective representations of infinite quivers. {A}pplications.
\newblock {\em Canad. J. Math.}, 61(2):315--335, 2009.

\bibitem{EJ95}
Edgar~E. Enochs and Overtoun M.~G. Jenda.
\newblock Gorenstein injective and projective modules.
\newblock {\em Math. Z.}, 220(4):611--633, 1995.

\bibitem{EJ11}
Edgar~E. Enochs and Overtoun M.~G. Jenda.
\newblock {\em Relative homological algebra. {V}olume 1}, volume~30 of {\em De
  Gruyter Expositions in Mathematics}.
\newblock Walter de Gruyter GmbH \& Co. KG, Berlin, extended edition, 2011.

\bibitem{EJ11a}
Edgar~E. Enochs and Overtoun M.~G. Jenda.
\newblock {\em Relative homological algebra. {V}olume 2}, volume~54 of {\em De
  Gruyter Expositions in Mathematics}.
\newblock Walter de Gruyter GmbH \& Co. KG, Berlin, 2011.

\bibitem{EHS13}
H.~Eshraghi, R.~Hafezi, and Sh. Salarian.
\newblock Total acyclicity for complexes of representations of quivers.
\newblock {\em Comm. Algebra}, 41(12):4425--4441, 2013.

\bibitem{Hol04}
Henrik Holm.
\newblock Gorenstein homological dimensions.
\newblock {\em J. Pure Appl. Algebra}, 189(1-3):167--193, 2004.

\bibitem{HJ16}
Henrik Holm and Peter J{\o}rgensen.
\newblock Cotorsion pairs in categories of quiver representations, 2016.
\newblock arXiv:1604.01517.

\bibitem{HLXZ17}
Wei Hu, Xiu-Hua Luo, Bao-Lin Xiong, and Guodong Zhou.
\newblock Gorenstein projective bimodules via monomorphism categories and
  filtration categories, 2017.
\newblock arXiv:1702.08669.

\bibitem{Hua13}
Zhaoyong Huang.
\newblock Proper resolutions and {G}orenstein categories.
\newblock {\em J. Algebra}, 393:142--169, 2013.

\bibitem{JKS16}
Bernt~Tore Jensen, Alastair~D. King, and Xiuping Su.
\newblock A categorification of {G}rassmannian cluster algebras.
\newblock {\em Proc. Lond. Math. Soc. (3)}, 113(2):185--212, 2016.

\bibitem{J07}
Peter J{\o}rgensen.
\newblock Existence of {G}orenstein projective resolutions and {T}ate
  cohomology.
\newblock {\em J. Eur. Math. Soc. (JEMS)}, 9(1):59--76, 2007.

\bibitem{JK11}
Peter J{\o}rgensen and Kiriko Kato.
\newblock Symmetric {A}uslander and {B}ass categories.
\newblock {\em Math. Proc. Cambridge Philos. Soc.}, 150(2):227--240, 2011.

\bibitem{JZ00}
Peter J{\o}rgensen and James~J. Zhang.
\newblock Gourmet's guide to {G}orensteinness.
\newblock {\em Adv. Math.}, 151(2):313--345, 2000.

\bibitem{Kel05}
G.~M. Kelly.
\newblock Basic concepts of enriched category theory.
\newblock {\em Repr. Theory Appl. Categ.}, (10):vi+137, 2005.
\newblock Reprint of the 1982 original [Cambridge Univ. Press, Cambridge;
  MR0651714].

\bibitem{Kva17}
Sondre Kvamme.
\newblock A generalization of the {N}akayama functor.
\newblock {\em arXiv:1611.01654v2}, pages 1--38, 2017.
\newblock Preprint.

\bibitem{LZ13}
Xiu-Hua Luo and Pu~Zhang.
\newblock Monic representations and {G}orenstein-projective modules.
\newblock {\em Pacific J. Math.}, 264(1):163--194, 2013.

\bibitem{LZ17}
Xiu-Hua Luo and Pu~Zhang.
\newblock Separated monic representations {I}: {G}orenstein-projective modules.
\newblock {\em J. Algebra}, 479:1--34, 2017.

\bibitem{MLan98}
Saunders Mac~Lane.
\newblock {\em Categories for the working mathematician}, volume~5 of {\em
  Graduate Texts in Mathematics}.
\newblock Springer-Verlag, New York, second edition, 1998.

\bibitem{NCha17}
Alfredo N\'ajera~Ch\'avez.
\newblock On {F}robenius (completed) orbit categories.
\newblock {\em Algebr. Represent. Theory}, 20(4):1007--1027, 2017.

\bibitem{OR70}
Ulrich Oberst and Helmut R\"{o}hrl.
\newblock Flat and coherent functors.
\newblock {\em J. Algebra}, 14:91--105, 1970.

\bibitem{Pop73}
N.~Popescu.
\newblock {\em Abelian categories with applications to rings and modules}.
\newblock Academic Press, London-New York, 1973.
\newblock London Mathematical Society Monographs, No. 3.

\bibitem{Pre17a}
Matthew Pressland.
\newblock Internally {C}alabi-{Y}au algebras and cluster-tilting objects.
\newblock {\em Math. Z.}, 287(1-2):555--585, 2017.

\bibitem{She16}
Dawei Shen.
\newblock A description of {G}orenstein projective modules over the tensor
  products of algebras, 2016.
\newblock arXiv:1602.00116.

\bibitem{Sto14}
Jan Stovicek.
\newblock Derived equivalences induced by big cotilting modules.
\newblock {\em Adv. Math.}, 263:45--87, 2014.

\bibitem{YL11}
Xiaoyan Yang and Zhongkui Liu.
\newblock Gorenstein projective, injective, and flat complexes.
\newblock {\em Comm. Algebra}, 39(5):1705--1721, 2011.

\end{thebibliography}
\bibliographystyle{plain} 

\end{document}